\RequirePackage{fix-cm}
\documentclass[smallextended]{svjour3}       % onecolumn (second format)
\smartqed  % flush right qed marks, e.g. at end of proof
\usepackage{graphicx}
\usepackage{picture}
\usepackage{curves}
\usepackage{bezier}
\usepackage{color}
\definecolor{grey}{rgb}{0.9,0.9,0.9}
\usepackage{geometry}
\usepackage[english]{babel}
\usepackage{amsmath}  
\usepackage{amssymb}
\usepackage{enumerate}
\usepackage{amsfonts}
\usepackage{a4wide}
\usepackage{geometry}
\usepackage{hyperref}
\usepackage{cite}

\numberwithin{equation}{section}

\newcommand{\ga}{\gamma}
\newcommand{\Ga}{\Gamma}
\newcommand{\Om}{\Omega}
\newcommand{\om}{\omega}

\newcommand{\E}{\mathbb E}
\newcommand{\R}{\mathbb R}
\newcommand{\Z}{\mathbb Z}

\newcommand{\cL}{\mathcal{L}}

\newcommand{\cX}{\mathcal{X}}
\newcommand{\cY}{\mathcal{Y}}
\newcommand{\cV}{\mathcal{V}}
\newcommand{\cH}{\mathcal{H}}
\newcommand{\cR}{\mathcal{R}}

\newcommand{\vep}{\varepsilon}

\newcommand{\SSEP}{\sc ssep}

\newcommand{\scj}{\text{\sc j}}
\newcommand{\scr}{\text{\sc r}}
\newcommand{\scl}{\text{\sc l}}
\newcommand{\scm}{\text{\sc m}}

\newcommand{\zr}{z}
\newcommand{\dD}{D}

\newcommand{\tLinter}{{\tilde\cL^{\text{\hskip-.3mm\scj}}_{\text{\hskip-.3mm\tiny\rm{inter}}}}}
\newcommand{\tLpart}{{\tilde L^{\text{\hskip-.3mm\scj}}_{\text{\hskip-.3mm\tiny\rm{part}}}}}

\newcommand{\Linter}{{\cL^{\text{\hskip-.3mm\scj}}_{\text{\hskip-.3mm\tiny\rm{inter}}}}}
\newcommand{\Lpart}{{L^{\text{\hskip-.3mm\scj}}_{\text{\hskip-.3mm\tiny\rm{part}}}}}
\newcommand{\Lssep}{{\cL_{\Delta}}}
\renewcommand{\ll}{{L^{\hskip-.3mm\scj}_{\scl}}}
\newcommand{\lr}{{L^{\hskip-.3mm\scj}_{\scr}}}
\newcommand{\Ll}{{\cL^{\hskip-.3mm\scj}_{\scl}}}
\newcommand{\Lr}{{\cL^{\hskip-.3mm\scj}_{\scr}}}
\newcommand{\tll}{{\tilde L^{\hskip-.3mm\scj}_{\scl}}}
\newcommand{\tlr}{{\tilde L^{\hskip-.3mm\scj}_{\scr}}}
\newcommand{\tLl}{{\tilde\cL^{\hskip-.3mm\scj}_{\scl}}}
\newcommand{\tLr}{{\tilde\cL^{\hskip-.3mm\scj}_{\scr}}}

\newcommand{\cm}{{F}}
\newcommand{\tcm}{{\hat F}}

\newcommand{\Llvep}{{\cL^{j\vep}_{\text{\hskip-.3mm\scl}}}}
\newcommand{\Lrvep}{{\cL^{j\vep}_{\text{\hskip-.3mm\scr}}}}

\newcommand{\one}{{\mathbf 1}}

\newcommand{\sqr}{\vcenter{
         \hrule height.1mm
         \hbox{\vrule width.1mm height2.2mm\kern2.18mm\vrule width.1mm}
         \hrule height.1mm}}

\newcommand{\dis}{\displaystyle}
\newcommand{\paf}[1]{{#1}} %no colors
\newcommand{\dpred}[1]{{#1}} %no colors

\journalname{Probability Theory and Related Fields}

\begin{document}

\title{Symmetric simple exclusion process with free boundaries
\thanks{To be
  published in Probability Theory and Related Fields. Final publication  available at
  \href{http://link.springer.com}{http://link.springer.com}.}
} 

%\subtitle{}

\titlerunning{SSEP{}   with free boundaries}        % if too long for running head

\author{Anna De Masi         \and
        Pablo A. Ferrari \and Errico Presutti %etc.
}

\authorrunning{De Masi, Ferrari, Presutti}
% if too long for running head

\institute{Anna De Masi \at
              Universit\`a di L'Aquila, 67100 L'Aquila, Italy \\
                            \email{demasi@univaq.it}           %  \\
%             \emph{Present address:} of F. Author  %  if needed
           \and
           Pablo A. Ferrari \at
              Universidad de Buenos Aires, Universidade de S\~ao Paulo\\
              \email{pferrari@dm.uba.ar}
	  \and
           Errico Presutti \at
              Gran Sasso Science Institute, 67100 L'Aquila, Italy\\
              \email{errico.presutti@gmail.com}
}

%\date{Received: date / Accepted: date}
% The correct dates will be entered by the editor

\maketitle

\begin{abstract}
  We consider the one dimensional symmetric simple exclusion process (SSEP) with
  additional births and deaths restricted to a subset of configurations where
  there is a leftmost hole and a rightmost particle.  At a fixed rate birth of
  particles occur at the position of the leftmost hole and at the same rate,
  independently, the rightmost particle dies.  We prove convergence to a
  hydrodynamic limit and discuss its relation with a free boundary problem.

\keywords{Hydrodynamic limit \and free boundary problems \and Stochastic inequalities}
% \PACS{PACS code1 \and PACS code2 \and more}
% \subclass{MSC code1 \and MSC code2 \and more}
\end{abstract}

\section{Introduction}
\label{intro}
\dpred{ A free boundary problem  in its simplest version}
is given by the linear heat equation in a domain $\Om$ which itself
changes in time with a law which depends on the same solution. \dpred{
In the Stefan problem} for instance \dpred{ the  heat equation is complemented
by Dirichlet boundary conditions while the}
local velocity of the points of the boundary are specified in terms of the local gradient
of the solution.

The purpose of this paper is to study a particle version of such free boundary
problems.  The linear heat equation is in our case replaced by the one
dimensional symmetric simple exclusion process, SSEP, in the set of
configurations having a rightmost particle and a leftmost hole. Furthermore, at
a given rate the rightmost particle dies and at the same rate, independently, a
birth of a particle occurs at the position of the leftmost hole. Since the leftmost
hole and the rightmost particle move, we call them free boundaries.

More precisely, Call $\eta\in\{0,1\}^\Z$ a particle configuration, think of
$\eta$ as the subset of $\Z$ occupied by particles and consider those $\eta$
having a rightmost particle located at $\scr(\eta):=\max(\eta)$ and a leftmost
hole located at $\scl(\eta):=\min(\Z\setminus\eta)$.  Let $(\eta_t)$ be the
Markov process performing symmetric simple exclusion process at rate $\frac12$
and such that the rightmost particle and the leftmost hole are killed at rate
$\scj$:
\[
\eta\to\eta\setminus\{\scr(\eta)\} \text{ and } \eta\to\eta\cup\{\scl(\eta)\}
\text{ at rate }\scj \text{ each}.
\]
Since particles are injected to the left and extracted from the right, $\scj$
can be seen as the average current of particles through the sistem.
It is well known that under a diffusive space and time  scaling the collective behavior
of the SSEP{}  is ruled by the linear heat equation \cite{demasipresutti}. We
perform the same scaling with a parameter $\vep$ such that time is $\vep^{-2}t$,
space $\vep^{-1}r$ and the killing is $\scj=\scj(\vep) = \vep j$, where $j$ is the macroscopic
current.

Consider a function $\rho:\R\to[0,1]$ identically zero to the right of
$\scr(\rho):=\sup\{r:\rho(r)>0\}<\infty$ and identically one to the left of
$\scl(\rho):=\inf\{r:\rho(r)<1\}>-\infty$ and continuous in
$(\scl(\rho),\scr(\rho))$. Call $\cR$ the set of functions with those
properties. We consider a macroscopic density $\rho\in\cR$ and ask the initial
configuration $\eta^{(\vep)}$ indexed by $\vep$ to approach the density $\rho$
as follows:
\begin{equation}
\label{e97}
 \lim_{\vep\to0} \sup_{a\le b}\Bigl|\vep\sum_{\vep x\in[a,b]}\eta^{(\vep)}(x) - \int_{a}^{b}  \rho(r)dr\Bigr|
 \;=\; 0
\end{equation}
Our main result is
\begin{theorem}
  \label{thm1}
Let $(\eta^{(\vep)}_t)$ be the  process with killing at rate $j\vep$ and with
initial configuration $\eta^{(\vep)}$ satisfying  \eqref{e97}. Then for $t\ge
0$ there
exists a function $\rho_t\in\cR$ such that $\rho_0=\rho$ and
\begin{equation}
  \label{e96}
\lim_{\vep\to0}  P\Bigl( \sup_{a\le b}\Bigl|\vep\sum_{\vep
  x\in[a,b]}\eta^{(\vep)}_{t\vep^{-2}}(x) - \int_{a}^{b}
\rho_t(r)dr\Bigr|>\gamma\Bigr)\;=\;0\qquad\text{for all }\gamma>0.
\end{equation}
\end{theorem}

\dpred{ We characterize the limit $\rho_t$ in terms of ``lower and upper barriers'',
the inequalities being in the sense of mass transport. The notion is defined
by introducing first a map from  density functions $\rho\in\cR$
to
functions $\phi(r|\rho)$, $r\in \mathbb R$, that we call ``interfaces'' and  then by saying that $\rho \le \rho'$
if $\phi(\cdot|\rho)\le \phi(\cdot|\rho')$.  See Section \ref{sec:4} where the notion is first defined at the particles level and then for densities in $\cR$.

The barriers are defined
by functions $\rho^{\delta,\pm}_t$, $\delta>0$, which satisfy
``discretized free boundary problems''. More specifically}
%
%
%
%We do not have a description of the evolution of $\rho_t$, but we can construct
%it as limit of approximate evolutions $\rho^\delta_t$ as $\delta>0$ goes to
%zero.  The functions $\rho^\delta_t$ satisfy a ``discretized Stefan problem'',
%namely
$\rho^{\delta,-}_t$ evolves according to the heat equation in the intervals
$[n\delta,(n+1)\delta)$ and at times $n\delta$ takes the rightmost portion
$j\delta$ of mass from the right and puts it on the left. Namely, define
the $\delta$-quantiles $\scr^\delta(\rho)$ and $\scl^\delta(\rho)$ by
\begin{equation}
  \label{e95}
  \int_{\scr^\delta(\rho)}^{\infty}\rho(r)dr =\delta,
  \qquad  \int^{\scl^\delta(\rho)}_{-\infty}(1-\rho(r))dr =\delta,
\end{equation}
Define also
\begin{equation}
  \label{d111}
  (\Ga^{\delta}\rho)(r):=
\begin{cases}
  1&\text{if } r\le \scl^{\delta}(\rho)\\
\rho(r) &\text{if }  \scl^{\delta}(\rho)< r< \scr^{\delta}(\rho)\\
0&\text{if } r\ge \scr^{\delta}(\rho),
\end{cases}
\end{equation}
and let $G_t(r,r')$ be the Gaussian kernel (see \eqref{k6}). Set $\rho^{\delta,-}_0=\rho$ and iteratively
\begin{eqnarray}
  \label{g88}
  \rho^{\delta,-}_t:= \begin{cases}
G_{t-n\delta}\rho^{\delta,-}_{n\delta},&\text{if } t\in[n\delta,(n+1)\delta),\quad
n=0,1,\dots\\ \Ga^{j\delta}
\rho^{\delta,-}_{n\delta-},&\text{if } t= n\delta,\quad n=1,2,\dots.
\end{cases}
\end{eqnarray}
which is well defined for
$\delta$ small enough. \dpred{ Define $\rho^{\delta,+}_t$ with the same
  evolution but with initial profile
  $\rho^{\delta,+}_0=\Ga^{j\delta}\rho$. We prove that
\[
\rho^{\delta,-}_t \le \rho_t \le \rho^{\delta,+}_t
\]
for any $\delta$ and any $t\in \delta \mathbb N$. We also prove that any function $\tilde \rho_t$
which satisfies the above inequality (for all $\delta$ and $t$ as above) must necessarily be equal to $\rho_t$
(uniqueness of separating elements).
The precise statement is in Theorem \ref{thm12}}.
In particular this allows to show that:
\begin{theorem}
  \label{teo2}
  Let $\rho\in\cR$ and $\rho_t$ be the evolution of Theorem
  \ref{thm1} with initial datum $\rho$. Let $\rho^{\delta,-}_t$ be the
  evolution \eqref{g88} with the same initial datum.  Then for any $a<b$ real
  numbers and for any $\delta>0$,
\begin{equation}
  \label{g87}
   \Bigl|\int_a^b\rho^{\delta,-}_t(r)dr -\int_a^b\rho_t(r)dr\Bigr|\;\le\; 2 j\delta,\qquad \forall t\ge 0
\end{equation}

\end{theorem}
 Theorems \ref{thm1} and \ref{teo2} are proved at the end of Subsection \ref{sub6.1}.

\medskip
A formal limit of \dpred{ our particle system}  leads to
conjecture that $\rho_t$ solves
  \begin{eqnarray}
    \label{0.1}
&& \frac{\partial \rho}{\partial t}= \frac 12  \frac{\partial^2 \rho}{\partial r^2},  \quad
r\in (\scl_t,\scr_t),\\&&  \scr_0, \scl_0, \rho(r,0) \;\text{given} \nonumber\\&&
\rho(\scl_t ,t)=1,\;\;\rho(\scr_t ,t)=0;\quad \frac {\partial\rho}{\partial r}(\scl_t ,t)= \frac {\partial\rho}{\partial r} (\scr_t ,t)=-2j \nonumber
  \end{eqnarray}
where
$\scr_t:=\scr(\rho(\cdot,t))$ and $\scl_t:=\scl(\rho(\cdot,t))$.  \dpred{ Indeed the density flux
$J(r)$ associated to the equation $\frac{\partial \rho}{\partial t}= \frac 12  \frac{\partial^2 \rho}{\partial r^2}$
is equal to
$J(r):=-\frac 12  \frac{\partial \rho}{\partial r}$ so that the last two conditions in \eqref{0.1}
just state that
the outgoing flux at $\scr_t$ is equal to the killing rate $j$ and that the
incoming flux at $\scl_t$ is equal to the birth rate $j$.

The free boundary problem \eqref{0.1} is not of Stefan type.  In fact in
\eqref{0.1} we impose
both Dirichlet and Neumann conditions as we prescribe the values of the function
and of its derivative at the boundaries, while in the classical Stefan problem
the  Dirichlet boundary conditions are complemented by assigning the
speed of the boundary (in terms of the derivative of the solution at the boundaries).
We can obviously recover the velocity of the boundaries from
a smooth solution $\rho(r,t)$ of \eqref{0.1} by
differentiating the identities $\rho(\scl_t ,t)=1$,
$\rho(\scr_t ,t)=0$, thus obtaining:}
  \begin{eqnarray}
    \label{0.2}
&& \frac{d \scl_t }{dt} = \frac{1}{2j}  \frac{\partial^2 \rho}{\partial r^2}(\scl_t ,t), \;\; \frac{d \scr_t }{dt} = \frac{1}{2j}   \frac{\partial^2 \rho}{\partial r^2}(\scr_t ,t)
  \end{eqnarray}
%The free boundary problem \eqref{0.1}--\eqref{0.2}
%does not have its traditional form,  in particular  \eqref{0.1} seems over-determined.
%To write it in a more familiar fashion let $u =\partial \rho/\partial r$, then
%  \begin{eqnarray}
%    \label{0.3}
%    && \frac{\partial u}{\partial t}= \frac 12  \frac{\partial^2 u}{\partial r^2}, \quad
%    r\in (\scl_t ,\scr_t ),\\
%    &&  \scr_0 , \scl_0 ,  u(r,0)     \;\text{given and such that} \int_{\scl_0 }^{\scr_0 }u(r,0)dr=1\nonumber\\
%    &&
%    u(\scl_t ,t)=-2j,\;\;u(\scr_t ,t)=-2j  \nonumber
%  \end{eqnarray}
%    \begin{eqnarray}
%    \label{0.4}
%    && \frac{d \scl_t }{dt} = \frac{1}{2j}  \frac{\partial u}{\partial r}(\scl_t
%    ,t), \;\; \frac{d \scr_t }{dt} = \frac{1}{2j}    \frac{\partial u}{\partial r}(\scr_t ,t)
%  \end{eqnarray}
%This is now
%the classical Stefan problem: a   diffusive equation with Dirichlet boundary conditions, \eqref{0.3},
%on an interval whose endpoints evolve with velocity determined by the spatial derivative of the solution, \eqref{0.4}.
%
%To reconstruct $\rho(r,t)$ from the solution $u$ of \eqref{0.3}--\eqref{0.4}
%we set:
%    \begin{eqnarray}
%    \label{0.5}
%\rho(r,t) : = - \int_{r}^{\scr_t } u(r',t) dr',\quad r\in [\scl_t ,\scr_t ]
%   \end{eqnarray}
% so that $\rho(\scr_t ,t)\equiv 0$ and $\rho(\scl_0 ,0)=1$.  By \eqref{0.4},
% \[
%\frac{d}{dt}\rho(\scl_t ,t)   =0 \quad \text{hence}\quad \rho(\scl_t ,t)=\rho(\scl_0 ,0)=1
%  \]
%Thus any classical solution  of \eqref{0.3}--\eqref{0.4} gives rise via \eqref{0.5}
%to a solution of \eqref{0.1}--\eqref{0.2}.

  \dpred{ The traditional way to study the hydrodynamic limit of particle
    systems is to prove that the limit law of the system is supported by weak
    solutions of a PDE for which existence and uniqueness of weak solutions
    holds true.  In our case this approach is problematic. While we know
    closeness to the heat equation away from the boundaries, we do not control
    the motion of the boundaries: we only know that they do not escape to
    infinity.  We are not aware of existence and uniqueness theorems for
    \eqref{0.1}, however in general in free boundary problems some assumptions
    of regularity on the motion of the boundaries is required, for instance
    Lipschitz continuity.

Our proof of hydrodynamic limit avoids this pattern as we prove
directly existence of the limit by squeezing
the particles density between lower and upper barriers which have a
unique separating element.  This suggests a variational approach to the
analysis of \eqref{0.1} based on a proof that its classical solutions
are also squeezed by the barriers, or more generally that
any limit of ``approximate solutions'' of \eqref{0.1} lies in between the barriers.
Such an approach has been carried in \cite{CDGP} for a simpler version of \eqref{0.1}
where particles are in the semi-infinite line with reflections at the origin, so that
there is a single free boundary $\scr_t$ where particles are killed at rate $j$,
while births occur at the origin, at same rate $j$. An extension to our case
is in preparation.}

A similar equation has been derived in
\cite{L} for a different interface process.
\dpred{ In that case an existence and uniqueness
theorem for the limit equation is proved to hold (see also reference therein).
Regularity of the free boundary motions
in \cite{L} follows from some monotonicity properties
intrinsic to the model and absent in our case.

Free boundary problems have also been derived
in \cite{Durrett} for particles evolutions via branching and in \cite{LandimValle} for a variant of the simple exclusion  process.

We have a proof that the limit evolution $\rho_t$ in Theorem \ref{thm1}
satisfies \eqref{0.1} in the very special case when $\rho_t$ is a time-independent
linear profile
with slope $-2j$.  Stationary solutions for free boundary problems have also been
studied in \cite{DPT} for non local evolutions in systems which undergo a phase transition.}
%
%
% see also reference therein for particle
%models for the Stefan problem and for the derivation of a stationary Stefan
%problem when phase transitions are present.

The model we study in this paper is inspired by previous works on Fourier law
with current reservoirs, \cite{DPTV1}--\cite{DPTV2}; its actual formulation came out from discussions with Stefano Olla to whom we are indebted. \dpred{ We are also indebted to F. Comets and
H. Lacoin for helpful comments and discussions and to a referee of PTRF for useful comments.}

In Section \ref{sec:2} we define the particle process and in Section \ref{sec:3} we prove the existence of
a unique invariant measure for the process as seen from the median.

Due to the non local nature of the birth-death process the usual techniques for
hydrodynamic limit fail.  To overcome this problem we use inequalities
based on imbedding the particle process in an interface process. The
relationship between the one dimensional nearest neighbors simple exclusion
process and the interface process is known since the seminal paper by Rost
\cite{rost}, where he established the hydrodynamics of the the asymmetric simple
exclusion process. In Section \ref{sec:4} we define the interface dynamics and
show that our particle process can be realized in terms of the interface
dynamics. We also introduce the delta interface processes which correspond to
the delta particle processes where the killing of particles and holes are
grouped together and occur only a finite number of times (uniformly in the
hydrodynamic limit) and give as limit the delta macroscopic evolution defined in
\eqref {g88}.

In Section \ref{sec:4} we also establish basic inequalities between the true and the
delta interface dynamics by giving a simultaneous explicit graphical
construction of all of them; that is, a coupling.  In Section \ref{sec:5} we
prove convergence in the hydrodynamic limit.  The proof uses that any limit
point of the true interface dynamics is squeezed in between two approximate
evolutions which depend on an approximating parameter $\delta>0$.  In Section
\ref{sec:6} we establish basic properties of the macroscopic evolution.  In
particular, we prove the existence of a stationary solution for the limit
evolution and use this result to prove that at any positive time the particle
density (in the limit evolution) is identically 0 and 1 outside of a compact.
In the last section we summarize the results.

\vskip2cm

\section{The free boundary \SSEP}
\label{sec:2}

The space of particle configurations is
\[
\cX:=\big\{\eta\in\{0,1\}^\Z:\sum_{x\ge 0}
\eta(x) <\infty,\; \sum_{x\le 0}(1-\eta(x))<\infty\big\}
\]
that is, for any $\eta\in\cX$ the number of particles to the right of the origin and
the number of holes to its left are both finite.
A configuration $\eta\in\cX$ has a rightmost particle located at $\scr(\eta)$
and a leftmost hole located at $\scl(\eta)$, where
\begin{equation}
	\label{a2.1}
\scr(\eta):= \max\{x\in\Z:\eta(x)=1\};\qquad
\scl(\eta):= \min\{x\in\Z:\eta(x)=0\}.
		\end{equation}
We define the \emph{median} $ \scm(\eta)$
of a configuration $\eta\in\cX$ as the unique $m\in\Z+\frac12$ such that
\begin{equation}
  \label{m45}
 \sum_{x>m} \eta(x) - \sum_{x<m}(1-\eta(x))=0,
\end{equation}
that is, the number of particles to the right of $\scm(\eta)$ is the same as the
number of holes to its left.  The definition is well-posed because for $\eta\in\cX$, as $m$
increases by one,  \eqref{m45} increases by one and goes to
$\pm\infty$ as $m$ goes to $\mp\infty$.
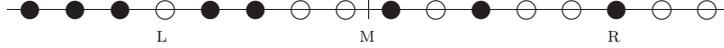
\begin{figure}
\setlength{\unitlength}{.6cm}
\begin{picture}(8,3)(-12,-2)
\linethickness{0.02mm}
\put(-8,-1){\line(1,0){16}}
\put(-7.5,-1){\circle*{.4}}
\put(-6.5,-1){\circle*{.4}}
\put(-5.5,-1){\circle*{.4}}
\put(-4.5,-1){\circle{.4}}
\put(-4.7,-1.7){$\scl$}
\put(-3.5,-1){\circle*{.4}}
\put(-2.5,-1){\circle*{.4}}
\put(-1.5,-1){\circle{.4}}
\put(-.5,-1){\circle{.4}}
\put(0.5,-1){\circle*{.4}}
\put(0,-1.2){\line(0,1){.4}}
\put(-0.2,-1.7){$\scm$}
\put(1.5,-1){\circle{.4}}
\put(2.5,-1){\circle*{.4}}
\put(3.5,-1){\circle{.4}}
\put(4.5,-1){\circle{.4}}
\put(5.5,-1){\circle*{.4}}
\put(5.3,-1.7){$\scr$}
\put(6.5,-1){\circle{.4}}
\put(7.5,-1){\circle{.4}}
\end{picture}
\caption{A typical configuration in $\cX$, black and white circles represent
  respectively particles and holes. $\scr$ is the position of the rightmost particle, $\scl$ the
position of the leftmost hole and $M$ is the median.}
\label{fig:figure1}
\end{figure}
We next define a family of dynamics indexed by a \emph{current} $\scj>0$.  The
particle dynamics is a (countable state) Markov process on $\cX$ whose generator
is
\begin{equation}
  \label{L00}
  \Lpart=L_0+\lr+\ll,
\end{equation}
where
\begin{equation}
  \label{l11}
L_0f(\eta) := \sum_{x\in\Z} \frac12[f(\eta^{x,x+1})-f(\eta)]
\end{equation}
with $\eta^{x,x+1}(y)=\eta(y)$ if $y\ne x,x+1$, $\eta^{x,x+1}(x)=\eta(x+1)$,
$\eta^{x,x+1}(x+1)=\eta(x)$; and
\begin{equation}
  \label{l12}
\ll f(\eta):= \scj\Big(f(\eta\cup\{\scl(\eta)\})-f(\eta)\Big);\quad
\lr f(\eta):= \scj\Big(f(\eta\setminus\{\scr(\eta)\})-f(\eta)\Big),
\end{equation}
where $\eta$ is identified with the set of occupied sites
$\{x\in\Z:\eta(x)=1\}$.
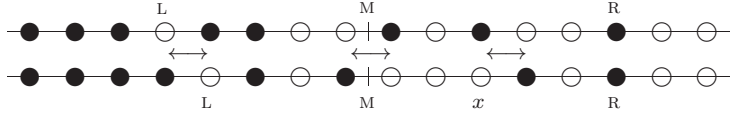
\begin{figure}
\setlength{\unitlength}{.6cm}
\begin{picture}(8,3)(-12,-3)
\linethickness{0.02mm}
\put(-8,-1){\line(1,0){16}}
\put(-7.5,-1){\circle*{.4}}
\put(-6.5,-1){\circle*{.4}}
\put(-5.5,-1){\circle*{.4}}
\put(-4.5,-1){\circle{.4}}
\put(-4.7,-.6){$\scl$}
\put(-4.45,-1.6){$\longleftrightarrow$}
\put(-3.5,-1){\circle*{.4}}
\put(-2.5,-1){\circle*{.4}}
\put(-1.5,-1){\circle{.4}}
\put(-.5,-1){\circle{.4}}
\put(-.4,-1.6){$\longleftrightarrow$}
\put(0.5,-1){\circle*{.4}}
\put(0,-1.2){\line(0,1){.4}}
\put(-0.2,-.6){$\scm$}
\put(1.5,-1){\circle{.4}}
\put(2.5,-1){\circle*{.4}}
\put(3.5,-1){\circle{.4}}
\put(2.6,-1.6){$\longleftrightarrow$}
\put(4.5,-1){\circle{.4}}
\put(5.5,-1){\circle*{.4}}
\put(5.3,-.6){$\scr$}
\put(6.5,-1){\circle{.4}}
\put(7.5,-1){\circle{.4}}

\put(-8,-2){\line(1,0){16}}
\put(-7.5,-2){\circle*{.4}}
\put(-6.5,-2){\circle*{.4}}
\put(-5.5,-2){\circle*{.4}}
\put(-4.5,-2){\circle*{.4}}
\put(-3.7,-2.7){$\scl$}
\put(-3.5,-2){\circle{.4}}
\put(-2.5,-2){\circle*{.4}}
\put(-1.5,-2){\circle{.4}}
\put(-.5,-2){\circle*{.4}}
\put(0.5,-2){\circle{.4}}
\put(0,-2.2){\line(0,1){.4}}
\put(-0.2,-2.7){$\scm$}
\put(1.5,-2){\circle{.4}}
\put(2.5,-2){\circle{.4}}
\put(2.3,-2.7){$x$}
\put(3.5,-2){\circle*{.4}}
\put(4.5,-2){\circle{.4}}
\put(5.5,-2){\circle*{.4}}
\put(5.3,-2.7){$\scr$}
\put(6.5,-2){\circle{.4}}
\put(7.5,-2){\circle{.4}}
\end{picture}
\caption{The effect of three SSEP jumps: the upper line is before and the
  bottom line is after the jumps. The jumps are between the positions $\scl$
  and $\scl+1$, $\scm-\frac12$ and $\scm+\frac12$ and  $x$ and $x+1$. None of
  these jumps change the position of $\scm$.
}
\label{fig:figure2}
\end{figure}

In other words, particles perform symmetric simple
exclusion with generator $L_{0}$ and, at rate $\scj$, the rightmost particle is
``killed'' and
replaced by a hole and at the same rate independently the leftmost hole is
killed and a particle is born at its place.
We omit the proof that the process is well defined at all times. Denote by
$(\eta_t)$ the process with generator $\Lpart$.

\medskip

\noindent
Denote by $B_t$ and $A_t$ the number of particles killed,
respectively born, in the time interval $[0,t]$.  By
definition, $A_t$ and $B_t$ are independent Poisson processes with rate~$\scj$.

\medskip

\begin{lemma}
\label{lemma1.1}
For any $\eta_0\in\cX$,
\begin{equation}
  \label{p1}
 \scm(\eta_t)=\scm(\eta_0)+A_t-B_t
\end{equation}
That is, the marginal distribution of the median of $\eta_t$ is a continuous time
symmetric nearest neighbor random walk on $\Z+\frac12$ with rate $\scj$ to jump
right and $\scj$ to jump left.
\end{lemma}

\begin{proof}
  Note that (a)
jumps due to the exclusion dynamics do not
 change $\scm(\cdot)$, (b)
   when the rightmost particle dies, $\scm(\cdot)$ decreases by 1 and (c)
   when the leftmost hole dies (and a particle appears in its place), $\scm(\cdot)$
   increases by~1. This shows \eqref{p1}. \qed
\end{proof}

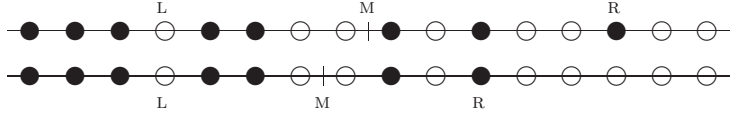
\begin{figure}
\setlength{\unitlength}{.6cm}
\begin{picture}(8,3)(-12,-3)
\linethickness{0.02mm}
\put(-8,-1){\line(1,0){16}}
\put(-7.5,-1){\circle*{.4}}
\put(-6.5,-1){\circle*{.4}}
\put(-5.5,-1){\circle*{.4}}
\put(-4.5,-1){\circle{.4}}
\put(-4.7,-.6){$\scl$}
\put(-3.5,-1){\circle*{.4}}
\put(-2.5,-1){\circle*{.4}}
\put(-1.5,-1){\circle{.4}}
\put(-.5,-1){\circle{.4}}
\put(0.5,-1){\circle*{.4}}
\put(0,-1.2){\line(0,1){.4}}
\put(-0.2,-.6){$\scm$}
\put(1.5,-1){\circle{.4}}
\put(2.5,-1){\circle*{.4}}
\put(3.5,-1){\circle{.4}}
\put(4.5,-1){\circle{.4}}
\put(5.5,-1){\circle*{.4}}
\put(5.3,-.6){$\scr$}
\put(6.5,-1){\circle{.4}}
\put(7.5,-1){\circle{.4}}

\put(-8,-2){\line(1,0){16}}
\put(-7.5,-2){\circle*{.4}}
\put(-6.5,-2){\circle*{.4}}
\put(-5.5,-2){\circle*{.4}}
\put(-4.5,-2){\circle{.4}}
\put(-4.7,-2.7){$\scl$}
\put(-3.5,-2){\circle*{.4}}
\put(-2.5,-2){\circle*{.4}}
\put(-1.5,-2){\circle{.4}}
\put(-.5,-2){\circle{.4}}
\put(0.5,-2){\circle*{.4}}
\put(-1,-2.2){\line(0,1){.4}}
\put(-1.2,-2.7){$\scm$}
\put(1.5,-2){\circle{.4}}
\put(2.5,-2){\circle*{.4}}
\put(2.3,-2.7){$\scr$}
\put(3.5,-2){\circle{.4}}
\put(4.5,-2){\circle{.4}}
\put(5.5,-2){\circle{.4}}
\put(6.5,-2){\circle{.4}}
\put(7.5,-2){\circle{.4}}
\end{picture}
\caption{The effect of killing the rightmost particle: the upperline is before
  the killing and the bottom line is after the killing. The position $\scr$ of the
  rightmost particle moves to the left by 3 (for this configuration) and $\scm$
  moves by 1 to the left (for any configuration). Analogously, the killing of
  the leftmost hole moves $\scm$ one unit to the right (not in the picture).}
\label{fig:figure3}
\end{figure}

\section{The process as seen from the median}
\label{sec:3}
Since $\scm(\eta_t)$ is a symmetric simple random walk the law of $\eta_t$ cannot
be tight but we show below that the process as seen from the median has a
unique invariant measure. Let \[\tilde\eta_t:= \theta_{\scm(\eta_t)-1/2}\eta_t,\]
where for $y\in \Z$, the translation $\theta_y:\cX\to\cX$ is the map
$(\theta_y\eta)(x) = \eta(x-y)$. Clearly $\tilde\eta_t\in\cX^0:=\{\eta\in \cX:
\scm(\eta)=1/2\}$ and we have that $(\tilde\eta_t)$ is a Markov process on
$\cX^0$ with generator
\begin{equation}
  \label{L0}
\tLpart= L_0+\tll+\tlr
\end{equation}
where for any $\eta\in \cX^0$,
\[
\tlr f(\eta) :=\scj\big(f(\theta_{-1}(\eta\setminus\{\scr(\eta)\}))-f(\eta)\big);\quad
\tll f(\eta):= \scj \big(f(\theta_1(\eta\cup\{\scl(\eta)\}))-f(\eta)]\big)
\]
Before stating the result we introduce some notation. For $\eta\in\cX$, define
\begin{equation}
\label{c1}
  N^{0}(\eta):=\sum_{x<\scr(\eta)}(1-\eta(x)),\qquad N^{1}(\eta):=\sum_{x>\scl(\eta)}\eta(x)
\end{equation}
the number of holes in $\eta$ to the
left of the rightmost particle  and the number of particles in $\eta$ to
the right of leftmost hole, respectively. Clearly
\begin{equation}
  \label{b17}
  N^{0}+N^{1}\;=\;\scr-\scl+1.
\end{equation}
Let $\eta^0\in\cX^0$ be the Heaviside configuration given by
\begin{equation}
  \label{b43}
  \eta^0(x) = \one\{x\le 0\}.
\end{equation}
Define $\psi$ on $\cX$ by
\begin{equation}
  \label{b41}
\psi(\eta):= \sum_{x<y}\; (1-\eta(x))\eta(y)
\end{equation}
Observe that $\psi(\eta)<\infty$ for all $\eta\in\cX$.
If $\eta\in\cX^0$, $\psi(\eta)$ is the number of jumps $10\to01$ needed to get
from $\eta^0$ to $\eta$. In particular, $\psi(\eta^0)=0$.
\medskip

\begin{theorem}
\label{thm1.2}
  For any $\scj>0$ the process $(\tilde\eta_t)$ has a unique invariant measure
  $\mu_{\scj}$ on $\cX^0$ and
  \begin{equation}
    \label{1.1}
  {\mu_{\scj}}[ \scr-\scl +1] = \frac1{2\scj}
  \end{equation}
\end{theorem}

\medskip

\begin{proof} We show that $\psi$ is a Lyapunov function for $(\tilde\eta_t)$
  by computing
  $\Lpart\psi$.
 Since $\psi(\theta_x\eta)=\psi(\eta)$ for all $x$ and $\eta$,
\begin{equation}
\label{1.3}
(\tll+\tlr)\psi(\eta)  =(\ll+\lr)\psi(\eta)  =  - {\scj}(N^{0}(\eta) - N^{1}(\eta))
=  - {\scj} \Big(\scr(\eta)-\scl(\eta)+1\Big)
\end{equation}
the second identity holds because when the rightmost particle
disappears, $\psi$ decreases by $N^{0}$,
the number of holes to the left of $\scr$.
Analogously, when the leftmost hole disappears, $\psi$ decreases by  $N^{1}$, the number of
particles to the right of $\scl$.  The third identity follows from \eqref{b17}.

Since  $\psi$ increases by one when there is a transition $10\to 01$ while it decreases by one
due to the opposite transition:
\begin{equation}
\label{1.2}
L_0\psi(\eta)= \frac 12 \sum_{x} \{\eta(x)(1-\eta(x+1))-(1-\eta(x))\eta(x+1)\} = \frac 12,
\end{equation}
because for any configuration $\eta\in\cX$, the number of pairs $10$ exceeds by one
the number of pairs $01$. The sums in \eqref{1.2} are finite for $\eta\in\cX$.

Call $\nu_t$ the law of $\tilde \eta_t$ starting from a configuration $\eta\in
\cX^0$.
For any $t\ge 0$ we have
${\nu_t}\psi<\infty$ and
\begin{equation}
\label{1.3a}
\frac 1t({\nu_t}\psi-\nu_0\psi) = \frac1t\int_0^t \nu_s[\Lpart\psi] ds= \frac 12
- \frac {\scj}t \int_0^t {\nu_s}[\scr-\scl+1]\,ds.
\end{equation}
Hence, calling $\mu_t:=  \frac 1t \int_0^t  \nu_s\,ds$,
\begin{equation}
\label{1.4}
{\mu_t} [\scr-\scl+1] \le  \frac 1{2\scj}+ \frac 1{\scj t}{\nu_0}\psi.
\end{equation}
Since for any $c$ we have $\# \{\eta\in \cX^0:\scr-\scl+1\le c\} <\infty$, then the
family $\{\mu_t,t\ge0\}$ of probabilities on $\cX^0$ is tight on $\cX^0$ and it has
therefore a limit point $\mu$.  $\mu$ is stationary by construction and unique
because the process is irreducible.

Identity \eqref{1.1} follows from \eqref{1.3a} by replacing $\nu_t$ by $\mu$,
but we need first to show that $\mu\psi<\infty$. Introduce a new Lyapunov
function $\psi_2$ on $\cX$ by setting
\[
\psi_2(\eta):= \sum_{x<y<z}\; [1-\eta(x)][1-\eta(y)]\eta(z)
\]
Then,
\begin{eqnarray*}
L_0\psi_2(\eta) &=& \frac 12 \Big( \sum_{x<y}\; [1-\eta(x)]\eta(y)[1-\eta(y+1)] -
\sum_{x<y}\; [1-\eta(x)][1-\eta(y)]\eta(y+1)\Big)\\
&=& \frac 12 \sum_{x<y}\; [1-\eta(x)][\eta(y)-\eta(y+1)]\; =\;\frac 12
\sum_{x}\; [1-\eta(x)]\eta(x+1)
\end{eqnarray*}
and
\begin{eqnarray*}
\lr\psi_2(\eta) &=& -\scj    \sum_{x<y<\scr}\; [1-\eta(x)] [1-\eta(y)]\; =\; -\frac{\scj}2
N^{0}(\eta)[N^{0}(\eta)-1] \\
\ll\psi_2(\eta) &=& -\scj  \sum_{\scl<y< z}\; [1-\eta(y)] \eta(z)
\;=\;-\scj\sum_{y<z}\; [1-\eta(y)] \eta(z) + \scj  \sum_{\scl<z}\;   \eta(z) \\
&=&  -\scj\psi(\eta)+ \scj N^{1}(\eta)
\end{eqnarray*}
Let $\nu_t$ as defined before \eqref{1.3a} and $\mu_t$ after it.
Then ${\nu_t}\psi_2<\infty$ and since
$\sum_{x}\; [1-\eta(x)]\eta(x+1)\le N^{0}(\eta)$,
\[
%\label{1.4}
\frac 1t\Big({\nu_t}\psi_2-{\nu_0}\psi_2\Big) \le - \frac {\scj}t \int_0^t
{\nu_s}\Big[-\frac {N^{0}}{2\scj} +\psi-N^{1}+
\frac{N^{0}}2(N^{0}-1)\Big]\,ds
\]
so that, since $N^{0} (N^{0}-1) \ge 0$,
\[
%\label{1.4}
\mu_t\psi \le \frac 1{\scj t} {\nu_0}\psi_2 +   {\mu_t}\Big[ \frac{N^{0}}{2\scj}
+N^{1}\Big] \le \text{Constant},
\]
by \eqref{b17} and \eqref{1.4}. Then $\mu\psi<\infty$ and setting $\nu_0=\mu$
in \eqref{1.3a} we get \eqref{1.1}.
\qed
\end{proof}

\medskip
The theorem says that under the invariant measure the average distance
between the rightmost particle and the leftmost hole is $(2\scj)^{-1}$. This
is in agreement with the Fick's law (the analogue of the Fourier law for mass
densities).  In fact Fick's law states that the stationary current $J$ flowing
in a system of length $\ell$ when at the endpoints the densities are
$\rho_{\pm}$ is:
\[
J=- \frac 12\; \frac{\rho_+-\rho_-}\ell
\]
$1/2$ being the particle mobility. In our case $J=\scj$,
$\rho_+=1$ and $\rho_-=0$ hence $\ell=(2\scj)^{-1}$.  The validity
of Fick's law in our case is however not completely obvious as the endpoints
$\scr(\eta_t)$ and $\scl(\eta_t)$ depend on time.

\setcounter{equation}{0}

\section{The interface process}
\label{sec:4}

\noindent
{\em The interfaces.}
Let the set of \emph{vertices} be
\begin{equation}
  \label{v1}
 \cV:=\{v= (v_1,v_2)\in \Z\times\Z^+ \text{ with
}v_1+v_2\text{ even}\}.
\end{equation}
and for a vertex $v=(v_1,v_2)\in\cV$ let $V_v:\Z\to\Z$ be the \emph{cone} with
vertex $v$ defined by
\begin{equation}
\label{2.1.0}
V_v(x):= |x-v_1| + v_2.
\end{equation}
Define the space of \emph{interfaces} as
\begin{eqnarray}
\nonumber
\cY_v&:=&\Big\{\xi\in\Z^\Z\,:\,|\xi(x)-\xi(x+1)|=1,\;x\in \Z;\;\#\{x:\xi(x) \neq
V_v(x)\}<\infty\}\\
\cY &:=& \cup_{v\in\cV} \cY_v
	\label{2.1}
\end{eqnarray}
That is, an interface in $\cY_v$ coincides with the cone $V_v$ for all but a
finite number of sites.  Interfaces share the property ``$x+\xi(x)$ even'' which
is conserved by the dynamics defined later.
For an interface $\xi\in \cY$ define
\begin{eqnarray}
  \label{b22a}
  \scl(\xi)&:=& \sup\{x\in\Z\,:\,\xi(x-y)=\xi(x) +y\;\text{ for all $y\ge 0$}\}\\ \label{b22b}
\scr(\xi)&:=& \inf\{x\in\Z\,:\,\xi(x+y)=\xi(x) +y\;\text{ for all $y\ge 0$}\}.
\end{eqnarray}
% \begin{figure}
% \label{fig:figure45}
% \setlength{\unitlength}{.9cm}
% \begin{picture}(8,10)(-9,-1.5)
% \linethickness{0.02mm}
% \put(0,-.5){\line(0,1){8}}
% \put(-1,-.50){\line(0,1){8}}
% \put(-2,-.50){\line(0,1){8}}
% \put(-3,-.50){\line(0,1){8}}
% \put(-4,-.50){\line(0,1){8}}
% \put(-5,-.50){\line(0,1){8}}
% \put(-6,-.50){\line(0,1){8}}
% \put(-7,-.50){\line(0,1){8}}
% \put(1,-.50){\line(0,1){8}}
% \put(2,-.50){\line(0,1){8}}
% \put(3,-.50){\line(0,1){8}}
% \put(4,-.50){\line(0,1){8}}
% \put(5,-.50){\line(0,1){8}}
% \put(6,-.50){\line(0,1){8}}
% \put(7,-.50){\line(0,1){8}}
%
% \put(-7.5,1){\line(1,0){15}}
% \put(-7.5,2){\line(1,0){15}}
% \put(-7.5,3){\line(1,0){15}}
% \put(-7.5,4){\line(1,0){15}}
% \put(-7.5,5){\line(1,0){15}}
% \put(-7.5,6){\line(1,0){15}}
% \put(-7.5,7){\line(1,0){15}}
%
% \put(-7.5,0){\line(1,0){15.5}}
% \put(0,0){\line(1,1){7}}
% \put(0,0){\line(-1,1){7}}
% \put(-5,-1){$\scl$}
% \put(6,-1){$\scr$}
% \put(0.2,-.4){$v$}
% %\put(-8,0){$v$}
% {\color{red}
% \thicklines
% \put(-5,5){\line(-1,1){2.5}}
% \put(-5,5){\line(1,1){1}}
% \put(-4,6){\line(1,-1){1}}
% \put(-3,5){\line(1,-1){1}}
% \put(-2,4){\line(1,1){1}}
% \put(-1,5){\line(1,1){1}}
% \put(0,6){\line(1,-1){1}}
% \put(1,5){\line(1,1){1}}
% \put(2,6){\line(1,-1){1}}
% \put(3,5){\line(1,1){1}}
% \put(4,6){\line(1,1){1}}
% \put(5,7){\line(1,-1){1}}
% \put(6,6){\line(1,1){1.5}}}
% \end{picture}
% \caption{In red a typical interface $\xi\in \cY_v$ and the cone  $V_v$.}
% \end{figure}
Any interface $\xi$ coincides with a cone $V_v$ outside the finite interval
$(\scl(\xi),\scr(\xi))$. The vertex $v=v(\xi)$ is the following function of
$\scr=\scr(\xi))$, $\scl=\scl(\xi)$, $\xi(\scl)$ and $\xi(\scr)$:
\begin{equation}
  \label{2.2}
v_1= v_1(\xi) :=\frac{\xi(\scl)-\xi(\scr)}{2}+ \frac{\scl+\scr}2,\;\;
v_2= v_2(\xi):= \frac{\xi(\scl)+\xi(\scr)}{2}+ \frac{\scl-\scr}2.
\end{equation}

\emph{Correspondence between interfaces and particle
configurations}. Given an interface $\xi$ we say that there is a particle
at $x$ if $\xi(x+1)<\xi(x)$ and that there is a hole at $x$ if $\xi(x+1)>\xi(x)$.
This defines the map $D:\cY\to \cX$ given by
\begin{equation}
  \label{2.2.1}
\eta(x)\equiv
\dD(\xi)(x) =  \frac 12 - \frac{\xi(x+1)-\xi(x)}2
\end{equation}
The map is clearly surjective but it is not injective  as $\dD$ is invariant
under uniform vertical shifts: $\dD(\xi+n) = \dD(\xi)$ for all $n\in \mathbb
Z$. However the extremes of the non conic part of $\xi$ correspond to the leftmost
hole and the rightmost particle of $\dD(\xi)$ and the absise of the vertex of the
cone containing $\xi$ corresponds to the median of $\dD(\xi)$. More precisely:

\vskip.5cm

\begin{lemma}
\label{lemma3}
 For any $\xi\in \cY$,
\[
 \scl(\dD(\xi))=\scl(\xi), \quad \scr(\dD(\xi))=\scr(\xi)-1,\quad \scm(\dD(\xi))=v_1(\xi)-1/2.
\]

\end{lemma}

\begin{proof}
%{\rm(i)}
  The first and second identities follows directly from the definitions. The
  third identity is trivially satisfied by any cone and its mapped particle
  configuration: $v_1(V_{(a,b)})=a$ and $\scm(\dD(V_{(a,b)}))=a-\frac12$. Any
  interface $\xi$ in the cone $V_{(a,b)}$ can be attained from $V_{(a,b)}$
  by a finite number of moves of the type:
\begin{equation}
  \label{a1}
  \xi=( \dots,z,z-1,z,\dots) \to (\dots,z,z+1,z,\dots)=\xi'
\end{equation}
The corresponding particle moves are
\begin{equation}
  \label{a2}
\eta=( \dots,1,0,\dots) \to (\dots,0,1,\dots)=\eta'
\end{equation}
and
\begin{equation}
  \label{a3}
  \dD(\xi)=\eta\quad\text{ if and only if }\quad \dD(\xi')=\eta'.
\end{equation}
Observing that $v(\xi)=v(\xi')$ and that $\scm(\dD(\xi))=\scm(\dD(\xi'))$,
we conclude that $\dD(\xi)$ has median $a-\frac12=v_1(\xi)$ for all $\xi\in \cY_{(a,b)}$.
\qed
\end{proof}

\vskip.5cm

\noindent
{\em Interface dynamics.} Let $(\xi_t)$ be the Markov process on $\cY$ with generator
\[
 \Linter:= \Lssep+ \Lr + \Ll,\quad \text{ with}
 \]
 \begin{equation}
   \label{m9}
   \Lssep f(\xi):=  \frac12\sum_{x\in\Z}\{f(\xi +\Delta_x\xi) -f(\xi)\},
 \end{equation}
\[\Lr f(\xi):= \scj \big[f(\max\{ \xi,V_{v(\xi)+(-1,1)}\} )-f(\xi)\big],\quad
 \Ll f(\xi):= \scj \big[f(\max\{\xi, V_{v(\xi)+(1,1)}\} )-f(\xi) \big],
\]
where $\Delta_x\xi(y):= (\xi(x+1)+\xi(x-1)-2\xi(x))\,\one\{y=x\}$.
The jumps of $\xi(x)$ due to $\Lssep$ occur only when $\xi(x)$ has the two
neighbors at equal height. The jump is up by 2 if the neighboring heights are
both above $\xi(x)$ and down by 2 if they are both below $\xi(x)$.  $\Lr$ acts
by changing the rightmost downward variation of $\xi$ into an upward one with
the interface to its right being a straight line with slope 1.  The cone
containing the updated interface is obtained from the previous cone by a
translation up by 1 and left by 1; see Figure \ref{fig:655}. A symmetrical
picture describes the action of $ \Ll$.  We denote by
\begin{equation}
  \label{2.4}
A_t=\#(\text{jumps due to $\Ll$ in $[0,t]$});\;\;\;\;B_t=\#(\text{jumps due to $\Lr$ in $[0,t]$})
\end{equation}
$A=(A_t)$ and $B=(B_t)$ are independent Poisson processes of intensity $\scj$.

The interface evolution induces via the map $\dD$ the particle evolution
described in Section \ref{sec:2}. \paf{Let $(\xi_t)$ be the interface process with
generator $\Linter$ and starting interface $\xi$ and define
\begin{equation}
  \label{pp1}
  \eta_t:=\dD(\xi_t)
\end{equation}}

%\vskip.5cm
\paf{
\begin{lemma}
\label{lemma4}
The particle process $(\eta_t)$ defined by \eqref{pp1} is Markov with generator
$\Lpart$, defined in \eqref{L00}. Moreover, if $\xi_0\in \cY_{(0,0)}$, then
\begin{equation}
  \label{2.5}
\xi_t(0)= 2 B_t + 2 \sum_{x\ge 0} \eta_t(x)
\end{equation}
\end{lemma}}
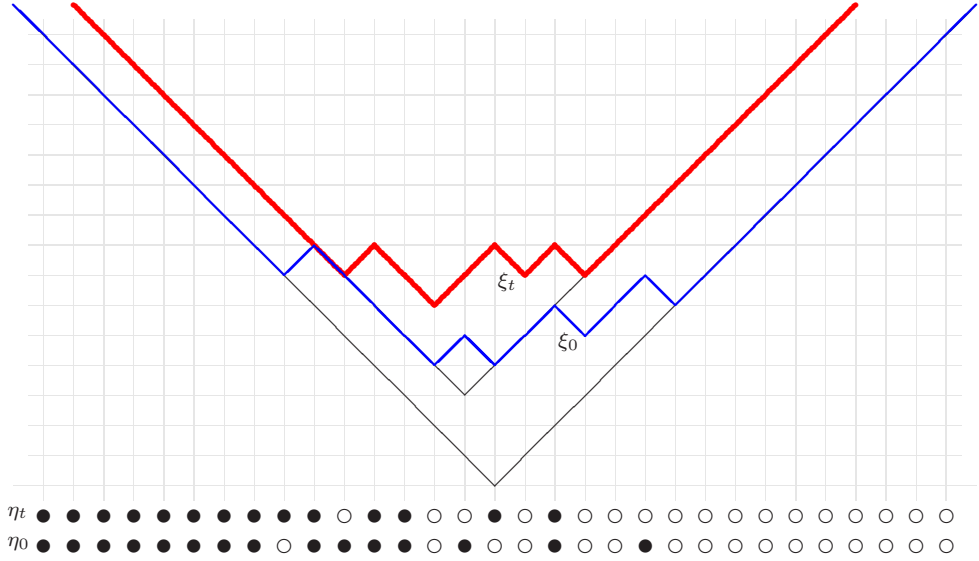
\begin{figure}
\setlength{\unitlength}{.4cm}
\begin{picture}(18,20)(-18,-2.5)
\linethickness{0.02mm}
{\color{grey}
\put(0,-.5){\line(0,1){16}}
\put(-1,-.50){\line(0,1){16}}
\put(-2,-.50){\line(0,1){16}}
\put(-3,-.50){\line(0,1){16}}
\put(-4,-.50){\line(0,1){16}}
\put(-5,-.50){\line(0,1){16}}
\put(-6,-.50){\line(0,1){16}}
\put(-7,-.50){\line(0,1){16}}
\put(-8,-.50){\line(0,1){16}}
\put(-9,-.50){\line(0,1){16}}
\put(-10,-.50){\line(0,1){16}}
\put(-11,-.50){\line(0,1){16}}
\put(-12,-.50){\line(0,1){16}}
\put(-13,-.50){\line(0,1){16}}
\put(-14,-.50){\line(0,1){16}}
\put(-15,-.50){\line(0,1){16}}
\put(1,-.50){\line(0,1){16}}
\put(2,-.50){\line(0,1){16}}
\put(3,-.50){\line(0,1){16}}
\put(4,-.50){\line(0,1){16}}
\put(5,-.50){\line(0,1){16}}
\put(6,-.50){\line(0,1){16}}
\put(7,-.50){\line(0,1){16}}
\put(8,-.50){\line(0,1){16}}
\put(9,-.50){\line(0,1){16}}
\put(10,-.50){\line(0,1){16}}
\put(11,-.50){\line(0,1){16}}
\put(12,-.50){\line(0,1){16}}
\put(13,-.50){\line(0,1){16}}
\put(14,-.50){\line(0,1){16}}
\put(15,-.50){\line(0,1){16}}
\put(-15.5,1){\line(1,0){31}}
\put(-15.5,2){\line(1,0){31}}
\put(-15.5,3){\line(1,0){31}}
\put(-15.5,4){\line(1,0){31}}
\put(-15.5,5){\line(1,0){31}}
\put(-15.5,6){\line(1,0){31}}
\put(-15.5,7){\line(1,0){31}}
\put(-15.5,8){\line(1,0){31}}
\put(-15.5,9){\line(1,0){31}}
\put(-15.5,10){\line(1,0){31}}
\put(-15.5,11){\line(1,0){31}}
\put(-15.5,12){\line(1,0){31}}
\put(-15.5,13){\line(1,0){31}}
\put(-15.5,14){\line(1,0){31}}
\put(-15.5,15){\line(1,0){31}}
\put(-16,0){\line(1,0){32}}
}
%new cone
%
\put(-1,3){\line(1,1){13}}
\put(-1,3){\line(-1,1){13}}
%
%new interface
%
{\color{red}
\linethickness{.6mm}%\thicklines
\curve(-5,7,-14,16)%\put(-5,5){\line(-1,1){2.5}}
\curve(-5,7,-4,8)%\put(-5,5){\line(1,1){1}}
\curve(-4,8,-3,7)%\put(-4,6){\line(1,-1){1}}
\curve(-3,7,-2,6)%\put(-3,5){\line(1,-1){1}}
\curve(-2,6,-1,7)%\put(-2,4){\line(1,1){1}}
\curve(-1,7,0,8)%\put(-1,5){\line(1,1){1}}
\curve(0,8,1,7)%\put(0,6){\line(1,-1){1}}
\curve(1,7,2,8)%\put(1,5){\line(1,1){1}}
\curve(2,8,3,7)%\put(2,6){\line(1,-1){1}}
\curve(3,7,12,16)}%\put(3,5){\line(1,1){1}}
%\curve(4,6,5.5,7.5)}%\put(4,6){\line(1,1){1.5}}
%\color{black}
%\curve(-1,1,-5,5)
%\curve(-1,1,3,5)
%}
%
%
%old cone:
\put(0,0){\line(-1,1){16}}  %old interface
\put(0,0){\line(1,1){16}}  %old interface
%
%old interface
{\color{blue}
{\thicklines
\put(-1,5){\line(1,-1){1}}
\put(-7,7){\line(-1,1){9}}
\put(-7,7){\line(1,1){1}}
\put(-6,8){\line(1,-1){1}}
\put(-5,7){\line(1,-1){1}}
\put(-4,6){\line(1,-1){1}}
\put(-3,5){\line(1,-1){1}}
\put(-2,4){\line(1,1){1}}
\put(-1,5){\line(1,-1){1}}
\put(0,4){\line(1,1){1}}
\put(1,5){\line(1,1){1}}
\put(2,6){\line(1,-1){1}}
\put(3,5){\line(1,1){1}}
\put(4,6){\line(1,1){1}}
\put(5,7){\line(1,-1){1}}
\put(6,6){\line(1,1){10}}
}}
\put(2.1,4.6){$\xi_0$}
\put(.1,6.6){$\xi_t$}
%\thicklines
%\put(5,7){\line(1,-1){1}}
%\put(6,6){\line(1,1){1.5}}
%
%new particle configuration
%
\put(-16.2,-1){$\eta_t$}
\put(-15,-1){\circle*{.4}}
\put(-14,-1){\circle*{.4}}
\put(-13,-1){\circle*{.4}}
\put(-12,-1){\circle*{.4}}
\put(-11,-1){\circle*{.4}}
\put(-10,-1){\circle*{.4}}
\put(-9,-1){\circle*{.4}}
\put(-8,-1){\circle*{.4}}
\put(-7,-1){\circle*{.4}}
\put(-6,-1){\circle*{.4}}
\put(-5,-1){\circle{.4}}
\put(-4,-1){\circle*{.4}}
\put(-3,-1){\circle*{.4}}
\put(-2,-1){\circle{.4}}
\put(-1,-1){\circle{.4}}
\put(0,-1){\circle*{.4}}
\put(1,-1){\circle{.4}}
\put(2,-1){\circle*{.4}}
\put(3,-1){\circle{.4}}
\put(4,-1){\circle{.4}}
\put(5,-1){\circle{.4}}
\put(6,-1){\circle{.4}}
\put(7,-1){\circle{.4}}
\put(8,-1){\circle{.4}}
\put(9,-1){\circle{.4}}
\put(10,-1){\circle{.4}}
\put(11,-1){\circle{.4}}
\put(12,-1){\circle{.4}}
\put(13,-1){\circle{.4}}
\put(14,-1){\circle{.4}}
\put(15,-1){\circle{.4}}
%
%old particle configuration
%
\put(-16.2,-2){$\eta_0$}
\put(-15,-2){\circle*{.4}}
\put(-14,-2){\circle*{.4}}
\put(-13,-2){\circle*{.4}}
\put(-12,-2){\circle*{.4}}
\put(-11,-2){\circle*{.4}}
\put(-10,-2){\circle*{.4}}
\put(-9,-2){\circle*{.4}}
\put(-8,-2){\circle*{.4}}
\put(-7,-2){\circle{.4}}
\put(-6,-2){\circle*{.4}}
\put(-5,-2){\circle*{.4}}
\put(-4,-2){\circle*{.4}}
\put(-3,-2){\circle*{.4}}
\put(-2,-2){\circle{.4}}
\put(-1,-2){\circle*{.4}}
\put(0,-2){\circle{.4}}
\put(1,-2){\circle{.4}}
\put(2,-2){\circle*{.4}}
\put(3,-2){\circle{.4}}
\put(4,-2){\circle{.4}}
\put(5,-2){\circle*{.4}}
\put(6,-2){\circle{.4}}
\put(7,-2){\circle{.4}}
\put(8,-2){\circle{.4}}
\put(9,-2){\circle{.4}}
\put(10,-2){\circle{.4}}
\put(11,-2){\circle{.4}}
\put(12,-2){\circle{.4}}
\put(13,-2){\circle{.4}}
\put(14,-2){\circle{.4}}
\put(15,-2){\circle{.4}}
\end{picture}
\caption{The thick red line represents the interface $\xi_t$; its corresponding
  particle configuration is $\eta_t=D(\xi_t)$. The narrow blue line represents
  $\xi_0$, the interface at time $0$; $\eta_0=D(\xi_0)$. Two particles and one
  hole have been killed and several particles have moved due to the exclusion
  dynamics; in particular the third particle moved to the place originally
  occupied by the second one and the second hole has moved three units to the
  left. Due to the killings the vertex of the cone containing $\xi_t$ has
  moved by $2(-1,1)+1(1,1)$ which amounts to three units up and one unit left so
  that $\xi_0\in\cY_{(0,0)}$ while $\xi_t\in\cY_{(-1,3)}$.}
\label{fig:655}
\end{figure}

%\medskip

\begin{proof}
  It just follows from the definitions of $(\xi_t)$ and $\dD$ that $(\dD(\xi_t))$ is
  Markov with generator $\Lpart$.  Calling $\scr\equiv
  \scr(\xi_t)=\scr(\eta_t)+1$ we get from
  \eqref{2.2.1}
\begin{equation*}
\xi_t(0)=\xi_t(\scr)-\scr +2\sum_{x\ge 0}\eta_t(x)
\end{equation*}
From \eqref{2.2}, $\xi_t(\scr)=\scr -(v_1-v_2)$ and since $v_1=A_t-B_t$ and $v_2=A_t+B_t$,
we get \eqref{2.5}. \qed
\end{proof}

\paragraph{Harris graphical construction} We
construct explicitly the interface process $(\xi_t)$ as a function of the initial
interface and of the  Poisson processes governing the different jumps.

We construct first the process with generator $\Lssep$ and later use it to
define the process with the moving boundary conditions.  The probability space
$( \Om,P)$ is the product of independent rate-$\frac12$ Poisson processes on $\mathbb R_+$
indexed by $\Z\times\{\uparrow,\downarrow\}$. A typical
element of $\Omega$ is $\omega = (\omega^\uparrow_x,\, \omega^\downarrow_x$,
$x\in \Z$). The Poisson points in $\omega^\uparrow_x$ and respectively
$\omega^\downarrow_x$ are called up-arrows and down-arrows, respectively.

We define operators $T_t:\Omega\times\cY\to\cY$ with $t\ge 0$ where
$(T_t(\omega,\xi))$ is the process with initial interface $\xi$ using the
arrows of $\omega$, as follows. We drop the dependence on $\omega$ and write
just $T_t\xi$, instead.  We can take $\om$ such that at most one arrow occurs at
any given time.

Set $T_0\xi=\xi$. Assume that $\xi_{t'}:=T_{t'}\xi$ is defined for all
$t'\in[0,s]$. %    and call
% \[
% \xi_{t-}:=\lim_{t'\nearrow t}T_{t'}\xi
% \]
Let $t$ be the first arrow after $s$ belonging to
$\omega^\uparrow_x\cup \omega^\downarrow_x$ for some $x$ such that
$\xi_s(x+1)=\xi_s(x-1)$. Since $\xi\in\cY$, there are a finite number of such
$x$ and $t-s>0$ a.s. These are the arrows \emph{involved} in the evolution at
time $s$.

Set
\[T_{t'}\xi= T_s\xi \text{ for }t'\in [s,t)\] and
(1) If $t$ is an up-arrow, then the interface at $x$ is set to
$\xi_t(x-1)+1$ no matter its value at $t-$ and does not change at the other
sites:
\[
T_t\xi(x):=T_{t-}\xi(x-1)+1;
%\xi_{t-}(x)+2 \,\one\{\xi_{t-}(x-1)=\xi_{t-}(x+1)>\xi_{t-}(x)\},
\quad T_t\xi(y) :=T_{t-}\xi(y),\;y\neq x.
\]
(2) Analogously, if $t$ is a down-arrow,%  at $x$ and the
  % height at $x$ is above the heights at $x\pm1$, then the
  % interface jumps down at $x$ by 2 and does not change at the other sites:
\[
T_t\xi(x):=T_{t-}\xi(x-1)-1;
%\xi_{t-}(x)-2 \,\one\{\xi_{t-}(x-1)=\xi_{t-}(x+1)<\xi_{t-}(x)\},
\quad T_t\xi(y) :=T_{t-}\xi(y),\;y\neq x.
\]
The reader can show that the process $T_t\xi$ so defined is Markov and evolves
with the generator $\Lssep$ with initial interface
$\xi$ at time $t=0$.

In the next definition we need to use the operator $T_t$ in different time
intervals. With this in mind we define
\begin{equation}
\label{e15}
T_{[s,t]}(\omega,\xi):=
T_{t-s}(\theta_{-s}\omega,\xi),
\end{equation}
where $\theta_s\omega$ is the translation by $s$ of the arrows in $\omega$.
That is, $T_{[s,t]}$ has the same distribution as $T_{t-s}$ but uses the arrows
in $\omega$ belonging to the interval $[s,t]$. We drop the dependence on
$\omega$ in the notation and write simply $T_{[s,t]}\xi$.

%\medskip
\paragraph{Generalizing the boundary conditions} Consider the partial order in
the vertex space $\cV$ given by
\begin{equation}
  \label{o20}
v\le v'\text{ if }V_{v}(x)\le
V_{v'}(x)\text{  for all }x\in\Z,
\end{equation}
so that vertex order corresponds to cone order.  Let $\zr=(\zr_t)$ with
$\zr_t\in\cV$ be a non decreasing path of vertices with finite number of finite
jumps in finite time intervals: $\|\zr_{t}-\zr_{t'}\|<\infty$ for all
$t<t'$. Let $T_{[s,t]}$ be the family of random operators governing the $\Lssep$
motion, defined in \eqref{e15}. Typically $\zr$ will be a function of the
Poisson processes $A$ and $B$ which are independent of the arrows $\omega$
used to define the operators $T_{[s,t]}$. We abuse notation and call $P$ the
probability associated to $\omega$, $A$ and $B$.

Let $0= s_0< s_1<\dots$ be the times of jumps of $\zr$. Define iteratively
$T^\zr_t\xi:=\xi$ for $t<0$ and
\begin{equation}
  \label{n12}
T^\zr_t\xi := \begin{cases}\max\{T^\zr_{t-}\xi,V_{\zr_t}\}&\text{if $t=s_n$},\\
T_{[s_n,t]}T^\zr_{s_n}\xi&\text{if $t\in(s_n,s_{n+1})$}.
\end{cases}
\end{equation}
So $(\xi^\zr_t)$ evolves with $\Lssep$ in the intervals $(s_n,s_{n+1})$  and at
times $s_n$ is updated to the maximum between the interface at time
$s_n-$ and the cone with vertex $\zr_{s_n}$.

Given an initial interface $\xi\in \cY$ and a number $\delta>0$, we will
consider the following choices for $\zr_t$, denoted by $O_t$, $R_t$ and
$\zr^{\delta,-}_t$ and $\zr^{\delta,+}_t$:
\begin{eqnarray}
O_t &:=& v(\xi),\quad \text{ for all }t;\label{o1t}\\
R_t&:=& v(\xi) +(A_t-B_t, A_t+B_t) \label{r1t00}\\
\zr^{\delta,-}_t&:=& R_{n\delta}, \qquad \zr^{\delta,+}_t\;:=\;
R_{(n+1)\delta}, \qquad t\in [n\delta,(n+1)\delta),\;\; n\ge 0.\label{d1t}
\end{eqnarray}
When the path is $(O_t)$, the cone does not move and the resulting process has
generator $\Lssep$. When the path is $(R_t)$ the process has generator $\Linter$.
The process with path $(\zr^{\delta,-}_t)$ records the increments of $(R_t)$ in the
intervals $[n\delta,(n+1)\delta)$ and takes the maximum of the interface at
the end of this interval and the cone with vertex $R_{n\delta}$. The process
with path $(\zr^{\delta,+}_t)$ records the increments of $(R_t)$ in the intervals
$[n\delta,(n+1)\delta)$ but takes the maximum between the interface at the
beginning of this interval and the cone with center $R_{(n+1)\delta}$. The
processes with paths \eqref{d1t} will be used later.

\paragraph{Monotonicity and attractivity}
Consider the
natural partial order in $\cY$ given by: $\xi\le \xi'$ if and only if
$\xi(x)\le\xi'(x)$ for all $x\in\Z$.  Use this order to define stochastic order
for random interfaces in $\cY$. \paf{If $\xi,\xi'$ are random, $\xi$ is stochastically
dominated by $\xi'$  if $Ef(\xi)\le Ef(\xi')$ for non decreasing
$f:\cY\to\R$. This is equivalent to the existence of a coupling
$(\hat\xi,\hat\xi')$ whose marginals have the same distribution as $\xi$ and
$\xi'$, respectively, such that $P(\hat\xi\le \hat\xi') =1$.}

Let $(\xi^1_t)$ and $(\xi^2_t)$ be two realizations of a stochastic process $(\xi_t)$ on $\cY$
with initial random interfaces $\xi^1$ and $\xi^2$, respectively. We say
that the process $(\xi_t)$ is \emph{attractive} if the following holds:
\begin{equation}
  \label{t1}
  \text{If } \xi^1\le \xi^2 \text{ stochastically, then } \xi^1_t\le \xi^2_t \text{
    stochastically, for
    all }t\ge 0.
\end{equation}
\begin{lemma}
  \label{t6}
  If $\xi\le\xi'$ then $T_t\xi\le T_t\xi'$ almost surely. As a consequence, the
  process with generator $\Lssep$ is attractive.
\end{lemma}
\begin{proof}
  Consider $\xi\le\xi'$ and call $\xi_t=T_t\xi$, $\xi'_t=T_t\xi'$, assume that
  $\xi_{t-}\le \xi'_{t-}$, $\xi_{t-}(x\pm1) = \xi'_{t-}(x\pm1)$ and that $t\in
  \omega^\uparrow_x$. Then $\xi_{t}(x)= \xi'_{t}(x)=\xi_{t-}(x-1)+1$ no matter the
  values $\xi_{t-}(x)$ and $ \xi'_{t-}(x)$. Analogous argument applies when
  $t\in \omega^\downarrow_x$. Hence an arrow does not change the order if the two
  interfaces coincide at $x\pm1$. If at least one of the neighbors of $x$ in
  $\xi_{t-}$ is different of the corresponding neighbor in $\xi'_{t-}$, a
  similar argument shows that no jump can break the domination. We have proven
  that if the the $\xi$ process is dominated by the $\xi'$ process just before
  an arrow, then the domination persist after the jump(s) produced by the
  arrow. Since the set of involved arrows is finite in any time interval, an
  iterative argument concludes the proof.
\qed\end{proof}
\paragraph{Remark} It is usual to realize the process with generator $\Lssep$ by
introducing only one rate-$\frac12$ Poisson process of marks $\omega_x, x\in \Z$
associated to each $x$ and updating $\xi(x)\to\xi(x-1)\pm 1\neq\xi(x)$ whenever a
$\omega_x$-mark appears provided $\xi(x-1)=\xi(x+1)$. This is indeed a
realization of the process, but order is not preserved: the jumps cross if two
interfaces coincide at $x\pm1$ and differ by 2 at $x$ at the updating time of
$x$. Using the up-arrows and down-arrows, only one of the interfaces jumps, and
order is preserved.
\begin{lemma}
  \label{m5}
  If $\xi\in \cY$, then $v(T_t\xi)=v(\xi)$ a.s. As a consequence, $V_{v(\xi)}\le
  T_t\xi$. Furthermore, if $\zr$ is a non decreasing path with
  $v(\xi)=\zr_0$, then $v(T^\zr_t\xi)=\zr_t$, which implies $V_{\zr_t}\le T^\zr_t\xi$.
\end{lemma}

\begin{proof} By the definition of $(T_t\xi)$, any time $\scl$ jumps due to an arrow,
  the opposite jump is performed by $\xi(\scl)$ and analogously, any jump of $\scr$ is
  replicated by $\xi(\scr)$. This implies
  \begin{equation}
    \label{m7}
    \scl(T_t\xi)+\xi(\scl(T_t\xi))=
  \scl(\xi)+\xi(\scl(\xi)),\quad \scr(T_t\xi)-\xi(\scr(T_t\xi))=
  \scr(\xi)-\xi(\scr(\xi))
  \end{equation}
  Putting these identities in the definition \eqref{2.2} of $v(\xi)$ and noting that
  the total number of jumps of $\scl$ and $\scr$ is a.s. finite (it is dominated by a
  Poisson process of rate $2$), we get $v(T_t\xi)=v(\xi)$.

  The fact that $v(T^\zr_t\xi)=\zr_t$ is true by definition at the $\zr$-events and by
  the first part of the lemma, it is true for $t\in[s,s')$ where $s$ and $s'$
  are successive $\zr$-events.
\qed\end{proof}

\begin{proposition}
\label{thm5p}
Let $\zr$ and $\zr'$ be non decreasing paths in $\cV$. 
\begin{equation}
  \label{n14}
\text{If }{\zr_t}\le {\zr'_t} \text{ for all }t\ge 0 \text{  and }\xi\le
\xi',\text{ then }T^{\zr}_{t}\xi\le
T^{\zr'}_{t}\xi' \text{ for all }t\ge0\quad \text{a.s.}.
\end{equation}
\end{proposition}
In particular, by taking $\zr=\zr'$, Proposition \ref{thm5p} says that for any non
decreasing path $\zr$, the process $(T^\zr_t\xi)$ is attractive.

\medskip
\begin{proof}\paf{
Since the domination is preserved in intervals with no $\zr$ or $\zr'$ events by
Lemma \ref{t6}, we only need to check that the inequality $T_t^{\zr}\xi\le
T_t^{\zr'}\xi$ is preserved when $\zr$ and $\zr'$ events occur. We thus suppose the
inequality is satisfied for all $s<t$ and this is evidently still true if $t$ is a
$\zr'$-event.  If instead $t$ is a $\zr$-event and not a $\zr'$ event,
$T^{\zr}_{t-}\xi\le T^{\zr'}_{t-}\xi'$ implies
\[
  T^{\zr}_t\xi\; =\;\max\{ T^{\zr}_{t-}\xi, V_{\zr_t}\} \;\le\; \max\{ T^{\zr'}_{t-}\xi',
  V_{\zr_t}\}  \;\le\; \max\{ T^{\zr'}_{t-}\xi',
  V_{\zr'_t}\}\;=\;  T^{\zr'}_t\xi,
\]
where the last inequality follows from  $V_{\zr_t}\le V_{\zr'_t}$ and the last
identity from the last inequality in Lemma~\ref{m5}.\qed}
\end{proof}

\paf{Let $\xi\in\cY$ and $v=(v_1,v_2)\in\cV$. Define  $\theta_v\xi$, the translation
by $v$ of $\xi$, by 
\[
\theta_v\xi(x):= \xi(x-v_1)-v_2.
\]
Recall the order \eqref{o20} in $\cV$. Taking $v,v'\in\cV$ and $\xi,\xi'\in
\cY$, the following statement is immediate.
\begin{equation}
  \label{7779}
 \hbox{ If }v\ge v'\hbox{ and } \xi\le \xi',\hbox{ then } \theta_v\xi\le\theta_{v'}\xi'.
\end{equation}
Call  $o:=(0,0)$ and take $v\in\cV$
satisfying $v\ge o$ and $\xi\in \cY_o$, then
\begin{equation}
  \label{t79}
  \theta_v\xi\le \xi \quad\text{ and } \quad\max\{\theta_v\xi,V_o\}\le \xi.
\end{equation}}

\paragraph{The interface process as seen from the vertex}\
Take $\xi\in \cY_o$ and let $z$ be a non decreasing path of vertices. Define $(\tilde T^\zr_t\xi)$, the interface process as seen from the vertex, by
\begin{equation}
  \label{170}
  \tilde T^\zr_t\xi(x) := \theta_{\zr_t}T^\zr_t\xi(x),
\end{equation}
 Of course $\tilde T^\zr_t\xi\in \cY_o$.

\paragraph{Monotonicity}
We show that if the initial interface $\xi'$ dominates $\xi$ and any jump of
$\zr'$ is \paf{dominated by a jump of $\zr$}, then the interface process as seen from $\zr'$
dominates the one as seen from $\zr$. More precisely,

\begin{proposition}
  \label{m3}
  Let $\zr$ and $\zr'$ be non decreasing paths on $\cV$ and $\xi\le \xi'$ be
  interfaces in $\cY_o$.
\begin{equation}
  \label{m4}
\begin{array}{c}
\text{If }\zr_t-\zr_{t-}\ge \zr'_t-\zr'_{t-} \text{ for all }t\ge 0,
\text{  then }\tilde T^{\zr}_{t}\xi\le
\tilde T^{\zr'}_{t}\xi' \text{ for all }t\ge0\quad \text{a.s.}.
\end{array}
\end{equation}
\end{proposition}

\begin{proof} \paf{Since by Lemma \ref{t6}, the domination is preserved in intervals with
  no $\zr$ or $\zr'$ events, it suffices to take care of those events.  Assume that
  $\tilde T^\zr_{t-}\xi \le \tilde T^{\zr'}_{t-}\xi'$ and that $t$ is a $z$ event, then
  \begin{equation}
    \label{m44}
   \tilde T^\zr_{t}\xi = \theta_{\zr_t-\zr_{t-}}\tilde T^\zr_{t-}\xi 
     %\;\le\;\theta_{\zr_t-\zr_{t-}}\tilde T^{\zr'}_{t-}\xi'
    \;\le\;\theta_{\zr'_t-\zr'_{t-}}\tilde T^{\zr'}_{t-}\xi'=\tilde T^{\zr'}_{t}\xi',
  \end{equation}
where the inequality holds by \eqref{7779}.} 
\qed\end{proof}

Let
\begin{equation}
  \label{vp1}
  \zr_t(\scj) := R_t= (A_t-B_t, A_t+B_t), 
\end{equation}
recalling that $A$ and
$B$ are independent Poisson processes of rate $\scj$. For each $\scj>0$ define
the interface process $(\tilde\xi_t)$ by
\begin{equation}
  \label{i22}
  \tilde\xi_t:=\tilde T^{\zr(\scj)}_{t}\xi
\end{equation}
Then, $(\tilde\xi_t)$ has generator is $\tLinter$ given by
\begin{equation}
  \label{m8}
  \tLinter= \Lssep+ \tLr + \tLl
\end{equation}
where $\Lssep$ was defined in \eqref{m9} and  the other generators govern the
updating with the maximum of the cone and the corresponding translation of the
origin:
\renewcommand{\arraystretch}{1.5}
\begin{equation}
  \label{m10}
  \begin{array}{ll}
\tLr f(\xi)&:= \scj \big[f(\max\{\theta_{(-1,1)}\xi, V_{o}\} )-f(\xi)\big],\\
\tLl f(\xi)&:= \scj \big[f(\max\{\theta_{(1,1)}\xi, V_{o}\} )-f(\xi)\big].
\end{array}
\end{equation}
For $\xi\in\cY_o$, the process $(\dD(\tilde\xi_t))$ has the same law as the
particle process $(\tilde\eta_t)$ defined in Section \ref{sec:3} with initial
particle configuration $\dD(\xi)\in\cX_0$.  The map $\dD:\cY_0\to\cX_0$
(interfaces with vertex in the origin to particle configurations with median
$-\frac12$) is bijective. Since the process $(\tilde\eta_t)$ has a unique invariant
measure on $\cX_0$, $(\tilde\xi^\zr_t)$ has a unique invariant measure
$\tilde\mu^\scj$ on $\cY_0$.
\paf{
\begin{corollary}
  \label{t12}
  The process $(\tilde\xi_t)$ is attractive. Furthermore, if
  $\tilde\xi_0=V_o$ the law of $\tilde\xi_t$ is non decreasing in $t$,
  is stochastically dominated by the invariant measure
  $\tilde\mu^{\scj}$ and converges to $\tilde\mu^{\scj}$ as
  $t\to\infty$.
\end{corollary}

\begin{proof}
  Attractivity follows by taking $\zr=\zr'$ in \eqref{m4}. As before use the
  notation $ \tilde T_{[s,t]}^z\xi$ to indicate that the evolution uses the
  Poisson processes of the Harris contruction in the time interval $[s,t]$ for
  both the interface evolution and the vertex evolution $z$. So that, for
  $s,t\ge 0$,  $\tilde
  T_{t+s}^z\xi$ has the same law as $\tilde T_{[-s,t]}^z\xi$. 
Since $V_o$ is minimal in
  $\cY_o$, using \eqref{m4} we get almost surely
\[
\tilde T_{[0,t]}^zV_o\le \tilde T_{[0,t]}^z\tilde T_{[-s,0]}^zV_o = \tilde
T_{[-s,t]}^zV_o . 
\]
This shows that the law of $\tilde\xi_t$ is stochastically non
decreasing. Take a random $\xi$ with law $\tilde\mu^{\scj}$. Then, 
\[
\tilde T_{[0,t]}^zV_o\le \tilde T_{[0,t]}^z \xi \sim \tilde\mu^{\scj}. 
\]
by invariance of $\tilde\mu^{\scj}$.  The convergence of the law of $\tilde
\xi_t$ to the unique invariant measure $\tilde\mu^{\scj}$ is routine for
countable state irreducible Markov processes. \qed
\end{proof}

\begin{corollary}
  \label{t13}
The invariant measures $\tilde\mu^\scj$ for the interface processes
$(\tilde\xi_{t})$ are stochastically ordered: 
\[
\text{If }\scj\ge \scj', \text{ then } \tilde\mu^\scj \le \tilde\mu^{\scj'}
\]
\end{corollary}
\begin{proof}
  Take Poisson processes $(A,A',B,B')$ such that $(A,A')$ and $(B,B')$ are
  independent. $A$ and $B$ have rate $\scj$ while $A'$ and $B'$ have rate
  $\scj'$ and $A\supset A'$, $B\supset B'$. In this way the vertex paths $z$ and
  $z'$ defined by \eqref{vp1} with $(A,B)$ and $(A',B')$, respectively, satisfy
  the conditions of Proposition \ref{m3}. This implies that $\tilde T_t^zV_o\le
  \tilde T_t^{z'}V_o$ almost surely for all $t$. Like in Corollary \ref{t12}, the coupled
  process $(\tilde T_t^zV_o, \tilde T_t^{z'}V_o)$ is stochastically non
  decreasing for the (partial) coordinatewise order and
  each coordinate is dominated by the respective invariant measure. This implies
  the existence of $\lim_{t\to\infty}(\tilde T_t^zV_o, \tilde T_t^{z'}V_o)$ in
  distribution. The coordinates of the limit are ordered and its marginals have
  distributions $\tilde\mu^\scj, \tilde\mu^{\scj'}$. \qed
\end{proof}
}

 \setcounter{equation}{0}
%\vskip2cm

\section{Hydrodynamic limit}
\label{sec:5}

It is well known that in the diffusive scaling limit (space scaled
as $\vep^{-1}$ and time as $\vep^{-2}$) the hydrodynamic limit of the SSEP process alone converges
to the linear heat equation:
\begin{equation}
\label{4.00}
\frac {\partial \rho_t}{\partial t} = \frac 12 \frac {\partial^2 \rho_t}{\partial r^2},\quad \rho_0=\rho
\end{equation}
Since the distance between the first hole and last particle is random, it is not clear a-priori that our model should be scaled diffusively as well.  From
 Theorem \ref{thm1.2} we know that at equilibrium the mean distance between the rightmost
particle and the leftmost hole is of order $\scj^{-1}$ and this suggest that together with the above diffusive scaling we should  scale $\scj$ proportionally to $\vep$.
Indeed we will prove that under such scaling limit the density
of the process converges as $\vep \to 0$ to a deterministic evolution.

\subsection{{\bf Results}}

\paragraph{Hydrodynamics of interfaces. }
 {\sl
  Initial configurations. } For each $\vep>0$ the interface evolution starts
from an interface $\xi^{(\vep)}$ such that:
\begin{equation}
\label{4.00.1}
\lim_{\vep\to 0} \sup_{x\in \Z} |\vep \xi^{(\vep)}(x)- \phi_0(\vep x)| = 0
\end{equation}
where $\phi_0(r)=r$ for all $r\ge \scr(\phi) >0$;  $\phi_0(r)=-r$ for all $r\le \scl(\phi)<0$
and $\phi_0$ is differentiable in
$(\scl(\phi),\scr(\phi))$ with derivative $\phi'_0$ such that $\dis{\sup_{r\in
    (\scl(\phi),\scr(\phi))
}|\phi'_0(r)| <1}$.

Fix the macroscopic current $j>0$, define $\scj(\vep):=j\vep$
and call
\begin{eqnarray}
  \label{l55}
  (\xi^{(\vep)}_t) &:=& \text{ interface process with generator }
  \Lssep+\Lrvep+\Llvep, \text{ starting from }\xi^{(\vep)}.
\end{eqnarray}
The following is the hydrodynamic limit for the interface process.

\vskip.5cm

\begin{theorem}
\label{thm6a}
There is
a  function $\phi_t(r)$, $t\ge 0$, $r\in \mathbb R$, so that for any $\ga>0$ and $t>0$:
 \begin{equation}
\label{4.00.2}
\lim_{\vep\to 0} P\Big[\sup_{x\in \Z} |\vep \xi^{(\vep)}_{\vep^{-2}t}(x)- \phi_t(\vep x)|
 \ge \ga\Big]= 0
\end{equation}

\end{theorem}
The proof will be given in subsection \ref{sub5.4} and
properties of  $\phi_t(r)$ will be discussed in Section \ref{sec:6}.

\paragraph{Hydrodynamics of particles.}  As we shall prove at the end of
subsection \ref{sub6.1}, the hydrodynamic limit for the
particle process Theorem \ref{thm1} is a corollary of Theorem~\ref{thm6a}.

\vskip.5cm

% Let $ P^{(\vep)}$ as in Theorem \ref{thm6a}, then there is  $\rho_t\in L^1(\mathbb R, [0,1])$ for all $t\ge 0$ such that the following holds.  For any $\ga>0$ and $t>0$:
%  \begin{equation}
% \label{4.00.4}
% \lim_{\vep\to 0}  P^{(\vep)}\Bigl[\sup_{a<b}
% \Bigl|\vep\sum_{\vep x\in[a,b]}\eta^{(\vep)}_{\vep^{-2}t}(x)
%  - \int_a^b
% \rho_t(r) dr \Bigr|
%  \ge \ga\Bigr]= 0
% \end{equation}

\vskip.5cm

The proof that \eqref{4.00} is the hydrodynamic limit for the SSEP is quite
simple because the correlation functions obey closed equations.  The addition of
birth-death processes spoils such a property but if the rates are cylinder
functions and are ``small'' (births and deaths happen at rate $\vep^{2}$) the
proofs carry over and the limit is a reaction diffusion equation as in
\cite{DFL}.  In our case the birth-death rates are not as small (because the
killing rates are of order $\vep$) but the main difficulty is that the killings
are highly non local functions of the configuration, since births and deaths
occur at the position of the leftmost hole and the rightmost particle,
respectively. This spoils completely an analysis based on the study of the
hierarchy of the correlation functions.

The way out is to use inequalities namely to sandwich the interface process
between two \emph{delta processes}, using Proposition \ref{m3}.  The
corresponding particle delta process behave as the exclusion dynamics in
macroscopic time intervals of lenght $\delta$ and the (accumulated) killings
occur at the extremes of those intervals.

\subsection{\bf The delta processes}
\label{d1p}
  {\em The delta interface process} $(\xi^{\delta,\pm}_t)$ is defined via
  \eqref{n12} and \eqref{d1t} by
\begin{equation}
  \label{e23}
  \xi_t^{\delta,\pm} := T^{\zr\pm}_t\xi,\text{ with }  \zr\pm:=\zr^{\pm\delta% \vep^{-2}
  }_{t}
\end{equation}
This process is obtained by patching together finitely many pieces
of the evolution with generator $\Lssep$ as explained after \eqref{d1t}. By
Proposition \ref{thm5p}, as
\[
z_t^{\delta,-}\le R_t\le z_t^{\delta,+},\qquad\hbox{for all }t\ge 0,
\]
\begin{equation}
  \label{r111}
  \xi_t^{\delta,-}\le \xi_t\le \xi_t^{\delta,+},\qquad\hbox{for all }t\ge 0.
\end{equation}

The {\em delta particle processes} are defined by using the map \eqref{2.2.1},
setting $\eta_t^{\delta,\pm}= \dD(\xi_t^{\delta,\pm})$.  In other words,
$(\eta_t^{\delta,-})$ evolves with the generator $L_0$ in the intervals
$[n\delta,(n+1)\delta)$, and $\eta_{n\delta}^{\delta,-}$ is obtained from
$\eta_{n\delta-}$ by removing its \paf{$B_{n\delta}-B_{(n-1)\delta}$ rightmost
particles and its $A_{n\delta}-A_{ (n-1)\delta}$} leftmost holes, where $A$
and $B$ are the independent Poisson processes with intensity $\scj$. The
interpretation of the process $(\eta_t^{\delta,+})$ is analogous but the removal
of particles and holes is done at the beginning rather than at the end of each
time interval.  For a particle configuration $\eta\in\cX$ and positive integers
$a$ and $b$ define the quantiles $\scl^a(\eta)$ and $ \scr^b(\eta)$ as the
lattice points satisfying
\begin{equation}
	\label{lr}
\sum_{x\ge \scr^b(\eta)}\eta(x)=b, \qquad \sum_{x\le \scl^a(\eta)}(1-\eta(x))=a.
\end{equation}
and define
\begin{equation}
 \Gamma^{a,b}(\eta)(x) = \eta(x)(1-\one_{x\ge \scr^b})+ (1-\eta(x))\one_{x\le
  \scl^a};
\end{equation}
this is the configuration obtained from $\eta$ by erasing particles and holes as
explained above.

We abuse notation and write $T_{[s,t]}\eta$ the evolution
$D(T_{[s,t]}\xi)$, with $\eta=D\xi$. Then the delta particle processes satisfy
the following: for $n=0,1,\dots$, 
\begin{eqnarray}
  \eta^{\delta,\pm}_t &=&
  T_{[n\delta,t]}\eta^{\delta,\pm}_{n\delta},\qquad  \hbox{for }t\in [n\delta,(n+1)\delta)\nonumber\\
  \eta^{\delta,-}_t &=& \Gamma^{A_{n\delta},B_{n\delta}}\eta^{\delta,-}_{t-},\qquad\hbox{for }
  t=n\delta\label{g93}\\
  \eta^{\delta,+}_t &=& \Gamma^{A_{(n+1)\delta},B_{(n+1)\delta}}\eta^{\delta,+}_{t-},\qquad\hbox{for }   t=n\delta\nonumber
\end{eqnarray}
where recall $A$ and $B$ are independent Poisson processes of parameter $\scj$.
\medskip

{\em The rescaled delta processes. } We consider a family of processes indexed
by $\vep$ by considering $\vep^{-2}\delta$ instead of $\delta$ and $\vep j$
instead of $\scj$.  We call $(\xi^{\vep,\delta,\pm}_t)$ and
$(\eta^{\vep,\delta,\pm}_t)$ the interface and particle processes so
obtained. Later we consider those processes with time rescaled by a factor $\vep^{-2}$
and space by a factor $\vep^{-1}$.

%We shall consider a family of processes $\{\xi^{\vep,\delta,\pm}_t\}$ with
%$\vep^{-2}\delta$ instead of $\delta$, $\vep j$ instead of $\scj$, $\vep^{-2}t$
%as time and $\vep^{-1}r$ as space and
%

For any fixed (macroscopic) $\delta>0$, we will prove the existence of
$\phi^{\delta,\pm}_t$, the limit as $\vep\to 0$ of
$\vep\xi^{\vep,\delta,\pm}_{t\vep^{-2}}$. We also prove that
$\phi^{\delta,\pm}_t$ are close to each other and that their difference vanishes
as $\delta\to 0$. Taking $\delta$ to zero, their common limit $\phi_t$ is then
the hydrodynamic limit of the rescaled evolution $\xi^{(\vep)}_t$ ---as this is
squeezed between $\xi^{\vep,\delta,-}_t$ and $\xi^{\vep,\delta,+}_t$.

While the above outline involves the interface process alone (which then implies
convergence for the particle process as well), yet the analysis of the limit as
$\vep \to 0$ of $(\xi^{\vep,\delta,\pm}_t)$ is more conveniently studied by
looking at the delta particle process $(\eta^{\vep,\delta,\pm}_t)$ and then
translating the results to the delta interface process.  We start from the
particle model defining first the corresponding approximate macroscopic density
delta-evolutions and then prove existence of the hydrodynamic limit for the
delta particle process.

  \vskip.5cm

\subsection{\bf The macroscopic delta-evolutions.}
\label{d88}

In this subsection we fix $\delta>0$ and define the macroscopic delta evolutions
of densities $\rho_t^{\delta,\pm}$ and interfaces $\phi^{\delta,\pm}_t$.

\emph{Preliminary results. } For any density $\rho:\R\to[0,1]$ define the
$\mathbb R_+ \cup \{+\infty\}$ valued functions
\[
\cm(r;\rho):=\int_r^\infty \rho(r') dr',\quad \tcm(r;\rho):=\int_{-\infty}^r (1-\rho(r'))dr'
\]
representing the \emph{mass} of $\rho$ to the right of $r$ and the \emph{antimass} to its
left. These are the macroscopic analogues of the number of particles,
respectively holes, to the right, respectively left, of $r$. We introduce two
disjoint subsets of densities called $\mathcal R$ and
$\mathcal U$.

Let $\mathcal R$ be the set of  densities $\rho\in  L^\infty(\mathbb R, [0,1])$
satisfying the following conditions: 
{\parindent 0pt
\parskip2mm

(i)\;  $\cm(0,\rho)= \tcm(0,\rho)<\infty$, that is mass to the right and
antimass to the left are finite and the origin is the \emph{median} of $\rho$.

(ii)\; $\rho$ has \emph{finite boundaries}. That is, there exist
$-\infty<\scl(\rho)\le\scr(\rho)<\infty$ such that
% $\cm(\scr(\rho);\rho)=0$
%and $\tcm(\scl(\rho);\rho)=0$. That is,
$\rho$ has \emph{finite support}
$[\scl(\rho),\scr(\rho)]$:
$\rho$ put no mass to the right of $\scr(\rho)$ and
no antimass to the left of $\scl(\rho)$.

(iii)\; $\rho$ is  \emph{continuous in the interior of the support}. That is, if
$\scl(\rho)<\scr(\rho)$ then $\rho(r)$ is continuous in
$(\scl(\rho),\scr(\rho))$ with values in $ (0,1)$.

 We call $\mathcal R^*$
the set  of $\rho\in \mathcal R$ such that $\scr(\rho)>0$.

\medskip
Let  $\mathcal U=\big\{u\in C(\mathbb R,
  (0,1)):\cm(0,u)= \tcm(0,u)<\infty\}$. This is the set of continuous
  densities with median~0.

Let $h\in\cR$ be the \emph{Heaviside density} defined by $h(r) = \one_{\{r\le
  0\}}$. Clearly $\scl(h)=\scr(h)=0$.
%$C_{\pm}(r;u) <\infty$  for any $r\in \mathbb R$ and
%\[
%C_+(r;u):=\int_r^\infty u(r') dr'<\infty,\quad C_-(r;u):=\int_{-\infty}^r (1-u(r'))dr' <\infty
%\]
}

\medskip

\begin{lemma}
\label{lemma8}
$\cm(r;u)<\infty$ and $\tcm(r;u)<\infty $ for all $r\in \mathbb R$ and
$u\in \mathcal U\cup  \mathcal R$.

If $\rho \in \mathcal R$ then $\scr(\rho)\ge 0$ and $\scl(\rho)\le 0$; the two
inequalities are strict unless $\rho=h$.
%is the Heaviside density.

If $\scr(\rho)>0$ then $\scl(\rho)<0$ and
the derivatives $\cm'(r;\rho)$ and $\tcm'(r;\rho)$ exist in $(\scl(\rho),\scr(\rho))$ where they are
respectively strictly negative and positive.

If $u\in \mathcal U$ then
$\cm,\,\tcm\in C^{1}(\mathbb R) $ and $\cm'(r;u)<0$ and $\tcm'(r;u)\ge 0$ for all $r\in
\mathbb R$.
%while if $u\in \mathcal R$ and $x(\rho)>0$, then $C_{\pm}(r;u)\in C^{1}((y(\rho),x(\rho))) $
%where  $C'_{\pm}(r;u)\lessgtr 0$.

For any $u\in \mathcal U\cup  \mathcal R$
and for any $\delta>0$
%
%the following holds.  $C_{\pm}(r;u)$ are  continuous functions of   $r\in \mathbb R$,
% $C_+(r;u)$ and  $C_-(r;u)$ being   non increasing and  respectively
% non decreasing functions of $r$.
there are unique points $\scr^\delta(u),\scl^\delta(u)$ such that
\begin{equation}
\label{d09}
  \tcm(\scl^\delta;u) = \delta;\qquad \cm(\scr^\delta;u) = \delta.
\end{equation}

If $\cm(0;u)  \gtreqless \delta$ then
$\scr^{\delta}(u) \gtreqless 0$ and  $\scl^{\delta}(u)  \lesseqgtr 0$.

\end{lemma}

\medskip

\begin{proof} Let $u\in  \mathcal U\cup  \mathcal R$ then for any $r\in \mathbb R$
\[
\cm(r;u) =\int_0^\infty u(r') dr' - \int_0^r u(r') dr'\le \cm(0;u) +|r| <\infty,
\]
with an analogous argument showing that also $\tcm(r;u)<\infty$.  If $\rho \in
\mathcal R$ then $\scr(\rho)\ge 0$ because if $\scr(\rho)< 0$ then $\cm(0;\rho)=0$. Since  $\cm(0;\rho)=\tcm(0;\rho)$, then $\tcm(0;\rho)=0$ and this gives a contradiction since $\scr(\rho)< 0$ and $\rho\equiv 1$ is not allowed.  Moreover $\scr(\rho)=0$ if and only if $\rho$ is
the Heaviside density because $\cm(0;\rho)=\tcm(0;\rho)$.

If $u\in \mathcal U$ then $\cm'(r;u)=-u(r)<0$ and $\tcm'(r;u)=1-u(r)>0$, by the
definition of $\mathcal U$. If $\rho\in \mathcal R$
and $\scr(\rho)>0$ then $\scl(\rho)<0$ and by the definition of $\mathcal R$, $\rho$
is continuous in $(\scl(\rho),\scr(\rho))$ and away from $0$ and $1$. Hence
$\cm'(r;\rho)=-\rho(r)<0$ and $\tcm'(r;\rho)=1-\rho(r)>0$ for all
$r\in(\scl(\rho),\scr(\rho))$.

Let $u\in \mathcal U\cup \mathcal R$.  By the monotonicity of $\cm(r;u)$ if
$\cm(0;u)>\delta$ then $\scr^\delta(u)>0$, while if $\cm(0;u)<\delta$ then
$\scr^\delta(u)<0$ with the analogous property for $\scl^\delta(u)$. \qed
\end{proof}

\medskip
\noindent

\noindent{\bf Definition of $\Ga^\delta$ and $G_t$}. We call $\Ga^\delta:\mathcal U \cup \mathcal
R\to \mathcal R$ the following map. If $ \scr^{\delta}(u)\le 0$ and therefore $
\scl^{\delta}(u)\ge 0$ we set $\Ga^\delta(u)=h$, the Heaviside density.  If
instead $ \scr^{\delta}(u)> 0$ and hence $ \scl^{\delta}(u)< 0$ we set
$\rho=\Ga^\delta(u)$ equal to $0$ for $r> \scr^{\delta}(u)$, equal to 1 for
$r< \scl^{\delta}(u)$ and equal to $u$ elsewhere.  In this latter case
$\rho=\Ga^\delta(u)\in\mathcal R^*$ (that is, $\scr(\rho)=\scr^{\delta}(u)>0$).
Thus $\Ga^\delta$ acts by removing a portion $\delta$ of mass from the right of
$\scr^{\delta}(u)$ and put it back to the left of $\scl^{\delta}(u)$.

Denote by $G_t$ the Gaussian kernel:
\begin{equation}
  \label{k6}
  G_t(r,r'):= \frac1{\sqrt{2\pi t}}{e^{-(r-r')^2/2t}}
\end{equation}
And write $G_t\rho(r)= \int dr' G_t(r,r')\rho(r')$. Recall that $G_t\rho$ is the
solution of the heat equation \eqref{4.00} with initial data $\rho$.
\medskip

\begin{lemma}
\label{lemma6}
Let $\rho \in \mathcal R\cup \mathcal U$. Then $G_t\rho\in\mathcal U$ for any
$t>0$.
%$\rho \in \mathcal R\cup \mathcal U$ and
 Moreover, calling
\begin{equation}
  \label{4.4.1}
  \mathcal U_{\delta}=\big\{ u\in \mathcal U:
  \scl^{\delta}(u)<0< \scr^{\delta}(u)\big\}
\end{equation}
for any $j>0$ there is $\delta(j)>0$ so that for any $u \in \mathcal R\cup \mathcal U$ and
$\delta<\delta(j)$, $G_\delta u\in \mathcal U_{j\delta}$ and therefore
$\Ga^{j\delta}(G_\delta u)\in \mathcal R^*$.
\end{lemma}

\begin{proof} Since $\rho \in \mathcal R\cup \mathcal U$, we have $\cm(0;\rho)
  <\infty$.  Then for any $t>0$:
  \begin{eqnarray*}
 && \hskip-1cm \cm(0;G_t\rho) =\int_0^{\infty} dr \int_{-\infty}^{+\infty} dr' G_t(r,r') \rho(r') \\
 & &\qquad \le \int_0^{\infty} dr \int_{0}^{+\infty} dr' G_t(r,r') \rho(r') +
  \int_0^{\infty} dr \int_{-\infty}^{0} dr' G_t(r,r')
  \;\le\;    \cm(0;\rho) +   c,
  \end{eqnarray*}
  where we used Fubini and $\int G_t(r,r') dr' =1$ to bound the first
  term.  Since $G_t\rho \in [0,1]$, $\cm(r;G_t\rho) <\infty$ for all $r$ (see the
  beginning of the proof of Lemma \ref{lemma8}).  An analogous argument shows
  that $\tcm(r;G_t\rho) <\infty$ for all $r$ and $t>0$ and, being the solution of
  the heat equation, $G_t\rho\in C^{\infty}$ for all $t>0$.  To prove that
  $G_t\rho\in \mathcal U$ it remains to show that $\cm(0;G_t\rho)=\tcm(0;G_t\rho)$
  for all $t>0$. \paf{Using the symmetry properties of $G_t$ and Fubini, for any $t\ge 0$ we have
  \[
  \cm(0;\rho_t)-\tcm(0;\rho_t)=\int_0^{\infty} dr G_t\rho(r) - \int_{-\infty}^0 dr
  [1-G_t\rho(r)]= 0.
  \]}%
% \dpred{  because by the properties of $\rho$ we can interchange derivative and integral and}
% by \eqref{4.00} both terms are equal to $\dis{\frac 12 \frac {\partial
%       \rho_t(r)}{\partial r}\big|_{r=0}}$. The identity
%   $\cm(0;\rho_t)-\tcm(0;\rho_t)=0$ then follows from its validity at time 0 and
%   by continuity.
%  For any $t\ge 0$ we have
%   \[
%   \frac{d}{d t} \Big( \int_0^{\infty} dr G_t\rho(r) - \int_{-\infty}^0 dr
%   [1-G_t\rho(r)]\Big)= 0
%   \]
% \dpred{  because by the properties of $\rho$ we can interchange derivative and integral and}
% by \eqref{4.00} both terms are equal to $\dis{\frac 12 \frac {\partial
%       \rho_t(r)}{\partial r}\big|_{r=0}}$. The identity
%   $\cm(0;\rho_t)-\tcm(0;\rho_t)=0$ then follows from its validity at time 0 and
%   by continuity.
  The last statement in Lemma \ref{lemma6} follows from the
  following inequalities
  \[ \int_{\sqrt \delta}^\infty dr \,G_\delta u(r) \ge \int_{\sqrt
    \delta}^\infty dr \,G_\delta h(r) \ge C\, \sqrt \delta,\quad C>0
  \]
  which is larger than $j\delta$ for $\delta$ small enough.  The first
  inequality follows from the fact that $u$ is stochastically larger than the
  Heaviside density $h$, in the sense that $u$ is obtained from $h$ by moving
  mass to the right. The last inequality follows from direct computations.
  \qed \end{proof}
\vskip.5cm

\paragraph{Density delta evolutions} We are finally ready to define the density delta
evolutions $\rho^{\delta,\pm}_t$. Restrict to $\delta\le \delta(j)$ as defined
in Lemma \ref{lemma6}, take the initial density $\rho\in\mathcal
R\cup \mathcal U$ and define iteratively $\rho^{\delta,-}_0=\rho$ and
\begin{equation}
  \label{r1t}
\rho^{\delta,-}_t\;:=\;
  \begin{cases}
       G_{t-n\delta}\rho^{\delta,-}_{n\delta},
  &t\in[n\delta,(n+1)\delta) \\
\Ga^{j\delta}\rho^{\delta,-}_{t-},& t=n\delta.
  \end{cases}
\end{equation}
The evolution
 $\rho_t^{\delta,+}$ is defined as $\rho^{\delta,-}_t$ but with
initial datum $\rho_0^{\delta,+}:=\Ga^{j\delta}(\rho)$. % The evolution
% $\rho^{\delta,-}_t$ is the one called $\rho^\delta_t$ in the Introduction.

\medskip
\paragraph{Interface delta evolutions}
The interface delta evolutions are defined as follows. Fix an initial interface
$\phi$ belonging to the cone with vertex at the origin and for $t\ge 0$ define
iteratively $\phi^{\delta,-}_0=\phi$, $\phi^{\delta,+}_0=\max\{\phi,V_{(0,\delta
  j)}\}$ and for $n\ge0$,
\begin{equation}
  \label{i8}
  \begin{array}{rclcl}
  \phi^{\delta,\pm}_t\;&:=&\;      G_{t-n\delta}\phi^{\delta,-}_{n\delta},
  &\text{if}&t\in[n\delta,(n+1)\delta) \\
\phi^{\delta,-}_t\;&:=&\; \max\{\phi^{\delta,-}_{t-},V_{(0,n
\delta j)}\},&\text{if}& t=n\delta.\\
\phi^{\delta,+}_t\;&:=&\; \max\{\phi^{\delta,+}_{t-},V_{(0,(n+1)
\delta j)}\},&\text{if}& t=n\delta.
  \end{array}
\end{equation}
We leave the proof of the following Lemma to the reader. It relates both
definitions.
\begin{lemma}
  \label{i1}
  The delta density evolutions $\rho^{\delta,\pm}_t$ defined in \eqref{r1t} and
  the delta interface evolutions $\phi^{\delta,\pm}_t$ defined in \eqref{i8} are
  related by
\begin{equation}
  \label{4.22}
 \phi^{\delta,\pm}_{t}(r)-  \phi^{\delta,\pm}_{t}(r') =  2 (r-r')  -
 2 \int_{r'}^r \rho^{\delta,\pm}_{t}(r'')dr''.
\end{equation}
The initial data are related by $\phi(0) = \int_{-\infty}^0 (1-\rho(r)) dr$
(this is the same as $\int^{\infty}_0 \rho(r) dr$ as $\rho\in\cR$); this fixes
the vertex of the cone of $\phi$ at the origin.
\end{lemma}

   \vskip.5cm

\subsection{{\bf Hydrodynamic limit of the delta particle process.}}

We now study the hydrodynamic limit for the particle delta process defined in
Subsection \ref{d1p}.
%We start describing the initial particle configurations.
We
introduce partitions $\mathcal D^{(\ell)}$ of $\Z$ into intervals $I^{(\ell)}$
of length $\ell$ where, denoting by $I^{(\ell)}_x$ the interval which contains
$x$, $I^{(\ell)}_0=[0,\ell-1]$ ($\mathcal D^{(\ell)}$ is now completely
specified). We take $\ell$ equal to the integer part of $\vep^{-\beta}$ with $\beta\in
(0,1)$, $\rho\in L^\infty( \mathbb R,[0,1])$ and (by an abuse of notation) we
write
 \begin{equation}
  \label{4.1}
 \mathcal A^{(\ell)}_x(\eta)= \frac 1{\ell} \sum_{y\in I^{(\ell)}_x} \eta(y),\quad
  \mathcal A^{(\ell)}_x(\rho)= \frac 1{\vep\ell} \int_{\vep I^{(\ell)}_x} \rho(r)dr
\end{equation}
not making explicit the dependence on $\vep$.

Introduce an accuracy parameter of the form~$\vep^{\alpha}$, $0<\alpha<\beta$;
the parameter $\beta \in (0,1)$ is fixed while $\alpha$ will change \dpred{ at each step of the
iteration scheme used in the sequel}.
Let $\mathcal G_{\vep,\alpha,\beta} (\rho)$be the set of
particle configurations which $(\vep,\alpha,\beta)$-\emph{recognize} the
macroscopic density $\rho \in \mathcal R$ defined by:
%$\cup
%\mathcal U$.  We further suppose that
%there are $y(\rho) \le 0 \le x(\rho)$ such that $\rho(x)=0$ for all $x>x(\rho)$ and $\rho(x)=1$ for all $y<y(\rho)$
%and set
 \begin{eqnarray}
  \label{4.2}
  &&\mathcal G_{\vep,\alpha,\beta} (\rho) =  \Big\{\eta \in \cX: |\vep
  \scl(\eta)-\scl(\rho)| + |\vep \scr(\eta)-\scr(\rho)| \le
  \vep^{\alpha},
  \sup_{x\in \Z\setminus \{I_\scr\cup I_\scl\}}| \mathcal A^{(\ell)}_x(\eta)-
  \mathcal A^{(\ell)}_x(\rho)| \le \vep^{\alpha} \Big\},\nonumber\\
&&
\qquad\qquad\qquad \ell=\text{ integer part of }\vep^{-\beta},\;\;0<\alpha<\beta<1,
\end{eqnarray}
where $ \scl(\eta)$ and $\scr(\eta)$ are defined in \eqref{a2.1}, and $I_\scl$
is the smallest $\mathcal D^{(\ell)}$ measurable interval which contains both
$\scl(\eta)$ and $\vep^{-1} \scl(\rho)$; $I_\scr$ is defined analogously with
reference to $\scr(\eta)$ and $\scr(\rho)$.

\begin{proposition}
\label{thm7}
Let $\rho \in \mathcal R$ and $\eta^{\vep,\delta,\pm}_0\in \mathcal
G_{\vep,\alpha,\beta}(\rho)$, then for any \dpred{ $\alpha'\in (0,\alpha)$ such that
$\alpha'< \min\{\frac \beta 2,1-\beta,\frac 14\}$}
%
% $\alpha'\in (0,\alpha)$ such that
%$\alpha'<1-\beta$ and $\alpha'< \beta/8$
the following holds: for any $k\ge 1$
there are coefficients $c_k$ so that
   \begin{equation}
  \label{4.6}
  P\Big[\eta^{\vep,\delta,\pm}_{\vep^{-2}\delta} \in \mathcal
  G_{\vep,\alpha',\beta}(\rho^{\delta,\pm}_\delta) \Big] \ge 1 -c_k\vep^k
  \end{equation}
\end{proposition}

\medskip
\noindent
{\bf Remark.} By iteration the result extends to any finite macroscopic time
interval and we \dpred{ also have}:

\vskip.5cm

\begin{corollary}
\label{thm10}
Under the same assumptions of Proposition \ref{thm7}, for any integer $m\ge 1$
and for any $k\ge 1$ there are coefficients $c_k$ so that
\begin{equation}
  \label{4.6a}
  P\Big[\bigcap_{n=1}^m\{\eta^{\vep,\delta,\pm}_{\vep^{-2}n\delta} \in \mathcal
  G_{\vep,\alpha',\beta }(\rho^{\delta,\pm}_{n\delta})\} \Big] \ge 1 -c_k\vep^k
\end{equation}

\end{corollary}

\vskip.5cm

To show Proposition \ref{thm7} we need to control the position
of the quantiles $\scr^a$ and $\scl^b$ of the process evolving with the
exclusion by the macroscopic time $\delta\vep^{-2}$. Here $a= A_{\vep^{-2}t}$,
$b= B_{\vep^{-2}t}$ \dpred{ which are} Poisson processes of parameter $\vep
j$. These bonds only depend on the exclusion dynamics governed by $L_0$.

\emph{Sharp convergence of the
  exclusion process to the solution of the heat equation}. Abusing notation
denote $(T_t\eta)$ the process in $\cX$ with initial configuration $\eta$ evolving
only with the exclusion generator $L_0$. The SSEP evolution is close to a linear
diffusion in the following sense: for any $n\ge 2$ there is $c_n$ so that for
any $t>0$
 \begin{equation}
  \label{4.7}
\sup_{(x_1,\dots,x_n)\in \Z_{\ne}^n}  | v(x_1,\dots,x_n;t)| \le c_n t^{-n/8}
\end{equation}
where $\Z_{\ne}^n$ is the set of all $n$-tuple of mutually distinct elements of $\Z$,
 \begin{equation}
  \label{4.8}
  v(x_1,\dots,x_n;t) = E\Big [ \prod_{i=1}^n \{T_t\eta(x_i) - u_t(x_i)\}\Big]
\end{equation}
\dpred{hereafter called $v$-functions,} and $u_t(x)$ solves the discretized heat equation
 \begin{equation}
  \label{4.9}
  \frac{du_t(x)}{dt}= \frac 12\, \Delta u_t(x)= \frac
  12\Big(u_t(x+1)+u_t(x-1)-2u_t(x)\Big),\quad u_0=\eta
\end{equation}
\eqref{4.7} is proved in \cite{demasipresutti}.
The solution of \eqref{4.9} is
 \begin{equation*}
%  \label{4.6}
u_{t}(x) = \sum_{y\in\Z} p_{t}(x,y) \eta(y)
\end{equation*}
$p_t(x,y)$ the transition probability kernel of the symmetric nearest-neighbors
random walk. The solution  of the heat equation starting from
$\rho$ is $G_t\rho$ (recall \eqref{k6}).
Thus, since $\eta\in \mathcal  G_{\vep,\alpha,\beta}(\rho)$
 \begin{equation}
  \label{4.10}
 |u_{\vep^{-2}t}(x)-G_t\rho(\vep x)|\le c' \Big( \vep t^{-1/2} +\vep^{1-\beta}t^{-1/2}+\vep^\alpha
 + \vep t^{-1/2}\vep^{-(1-\alpha)}
 \Big)\le c
 \Big( \vep^{1-\beta}t^{-1/2}+\vep^\alpha\Big)
\end{equation}
\dpred{ The proof of the inequality is done by changing $u_{\vep^{-2}t}(x)$
into  $G_t\rho(\vep x)$ in successive steps:}

\begin{itemize}
\item \dpred{ Replace $p_{\vep^{-2}t}(x,y)$ by $G_t(\vep x, \vep y)$. By
the local central limit theorem the error is bounded by the first term on the right hand side of \eqref{4.10}.}

\item  \dpred{  Replace $G_t(\vep x, \vep y)$ by its average in the intervals $\vep I_z^{(\ell)}$
 of length $\vep^{1-\beta}$, hence the second term on the right hand side of \eqref{4.10}.}

\item \dpred{  The contribution of the difference between averages of $\eta$ and $\rho$
in good intervals (i.e\ those not in $I_\scl\cup I_\scr$) is bounded by $\vep^\alpha$,
the
contribution of the intervals in $I_\scl\cup I_\scr$ by $\vep t^{-1/2} \vep^{-(1-\alpha)}$.}

\item  \dpred{  We finally reconstruct in each interval $\vep I_z^{(\ell)}$ the correct term
from $G_t\rho(\vep x)$ with an error given again by
the second term on the right hand side of \eqref{4.10}.}

\end{itemize}
%The first term takes into account the difference between $p_t$ and $G_t$, the
%bound follows from .  The second terms takes into
%account the variations of $G_t$ in an interval of length $\vep^{1-\beta}\ge \vep
%\ell$.  The contribution of the difference between averages of $\eta$ and $\rho$
%in good intervals (i.e\ those not in $I_\scl\cup I_\scr$) is bounded by $\vep^\alpha$; the
%contribution of the intervals in $I_\scl\cup I_\scr$ by $\vep t^{-1/2} \vep^{-(1-\alpha)}$.
\medskip

\dpred{ {\sl
Bounds on $|\mathcal A^{(\ell)}_x(T_{\vep^{-2}t}\eta)-
  \mathcal A^{(\ell)}_x(\rho_t)|$}. For any $x\in \mathbb Z$
  \begin{eqnarray}
  \label{4.11.00}
  |\mathcal A^{(\ell)}_x(T_{\vep^{-2}t}\eta)-
  \mathcal A^{(\ell)}_x(\rho_t)| &\le& |\mathcal A^{(\ell)}_x(T_{\vep^{-2}t}\eta)-
 \mathcal A^{(\ell)}_x(u_{\vep^{-2}t})|+|\mathcal A^{(\ell)}_x(u_{\vep^{-2}t})-
  \mathcal A^{(\ell)}_x(\rho_t)| \nonumber\\ &\le& |\mathcal A^{(\ell)}_x(T_{\vep^{-2}t}\eta)-
 \mathcal A^{(\ell)}_x(u_{\vep^{-2}t})|+ c
 \Big( \vep^{1-\beta}t^{-1/2}+\vep^\alpha\Big),
  \end{eqnarray}
 by \eqref{4.10}. 
We are going to show that for any integer $n$,}
 \begin{equation}
  \label{4.11}
\dpred{E\Big[ |\mathcal A^{(\ell)}_x(T_{\vep^{-2}t}\eta)-
  \mathcal A^{(\ell)}_x(u_{\vep^{-2}t})|^{2n}\Big] \le c \Big( \vep^{\beta n}   +
  [t\vep^{-2}]^{-n/4}\Big)}
\end{equation}
\dpred{Proof of \eqref{4.11}:}

\begin{itemize}
\item \dpred{  We expand $|\mathcal A^{(\ell)}_x(T_{\vep^{-2}t}\eta)-
  \mathcal A^{(\ell)}_x(u_{\vep^{-2}t})|^{2n}$ getting a sum of products
  of factors $\eta_{\vep^{-2} t}(z) - u_{\vep^{-2}t}(z)$.}

\item Each term of the form $\Big(\eta_{\vep^{-2} t}(z) -
  u_{\vep^{-2}t}(z)\Big)^k$ with $k>1$ can be rewritten as $c +
  c'(\eta_{\vep^{-2} t}(z) - u_{\vep^{-2}t}(z))$ with constants $c$ and $c'$ not
  depending on $\eta$. $c$ and $c'$ depend on the value of $u_{\vep^{-2}t}(z)$
  but that each of them is always smaller (in absolute value) than one.

\item \dpred{  Thus $E\big[|\mathcal A^{(\ell)}_x(T_{\vep^{-2}t}\eta)-
  \mathcal A^{(\ell)}_x(u_{\vep^{-2}t})|^{2n}\big]$ is a sum of product of constants times $v$-functions.
  We then use \eqref{4.7} to get \eqref{4.11}.}

\end{itemize}

Let $\ga>0$, then  since $\eta\in \mathcal  G_{\vep,\alpha,\beta}(\rho)$,
 \begin{equation*}
 % \label{4.6}
   \sum_{x\ge \vep^{-1-\ga}}P[ T_{\vep^{-2}t}\eta(x)=1] \le c'_k\vep^k,\quad
   \sum_{x\le -\vep^{-1-\ga}}P [ T_{\vep^{-2}t}\eta(x)=0] \le
   c'_k\vep^k
\end{equation*}
 As a consequence
  \begin{equation}
   \label{4.12}
   P\Big[  \scr(T_{\vep^{-2}t}\eta)\le  \vep^{-1-\ga};
   \;\;\scl(T_{\vep^{-2}t}\eta)\ge -\vep^{-1-\ga}\Big]\ge 1- c''_k\vep^k
 \end{equation}
By the hypotheses on $\rho$,
 \begin{equation}
  \label{a5.24}
\int_{r\ge \vep^{-\ga}}G_t\rho(r) dr \le c'_k\vep^k,\quad \int_{r\le
  -\vep^{-\ga}} dr [1- G_t\rho(r)] \le c'_k\vep^k.
\end{equation}
\dpred{ which proves that
  \begin{equation}
  \label{4.5.bb}
P\Big[\sup_{|x|\ge
\vep^{-1-\ga}}  | \mathcal A^{(\ell)}_x(T_{\vep^{-2}t}\eta)-
  \mathcal A^{(\ell)}_x(G_t\rho)| \le c'_k\vep^k\Big] \ge 1 -c'''_k\vep^k.
\end{equation}
We shall use \eqref{4.5.bb} to prove that for any $\alpha'$ as in Proposition \ref{thm7}
  \begin{equation}
  \label{4.5}
P\Big[\sup_{x\in \Z}| \mathcal A^{(\ell)}_x(T_{\vep^{-2}t}\eta)-
  \mathcal A^{(\ell)}_x(G_t\rho)| \le \vep^{\alpha'}\Big] \ge 1 -c_k\vep^k.
\end{equation}
%By \eqref{4.12} for $|x|\ge
%\vep^{-1-\ga}$, $|\mathcal A^{(\ell)}_x(T_{\vep^{-2}t}\eta)- \mathcal
%A^{(\ell)}_x(G_t\rho)|\le c'_k\vep^k$
By \eqref{4.5.bb} and \eqref{4.11.00} it suffices to prove that for any $\alpha'$ as above
  \begin{equation}
  \label{4.5cc}
P\Big[\sup_{|x|\le \vep^{-1-\ga}}| \mathcal A^{(\ell)}_x(T_{\vep^{-2}t}\eta)-
  \mathcal A^{(\ell)}_x(u_{\vep^{-2}t})| \le \vep^{\alpha'}\Big] \ge 1 -c_k\vep^k.
\end{equation}
which follows using the
Chebishev's inequality with power $2n$ for $n$ sufficiently large and  \eqref{4.11},  because
$\alpha'< \min\{\frac \beta 2,\frac 14\}$.}
%
% %and $\ga\le \beta+\alpha'$
%then for any $k\ge 1$ there are coefficients $c_k$ so that}
%  \begin{equation}
%  \label{4.5}
%P\Big[\sup_{x\in \Z}| \mathcal A^{(\ell)}_x(T_{\vep^{-2}t}\eta)-
%  \mathcal A^{(\ell)}_x(G_t\rho)| \le \vep^{\alpha'}\Big] \ge 1 -c_k\vep^k.
%\end{equation}

\medskip

\emph{Quantile bounds.} \dpred{ To complete the proof of \eqref{4.6} we fix $\alpha'$ as in Proposition
\ref{thm7} and  take $\alpha'' < \min\{\frac \beta 2,1-\beta,\alpha,\frac 14\}$ such that
$\alpha'' >\alpha'$.  Then there is a positive $\ga$ such that $\alpha''-2\ga >\alpha'$.

We now fix $t=\delta$ and use \eqref{4.5} (with $\alpha''$) }
to obtain bounds for the quantiles
defined in \eqref{d09} and \eqref{lr}.  Recalling
\eqref{d09} let
\[
\scr':= \vep^{-1}\scr^{j\delta}(G_\delta\rho) + \vep^{-1+\alpha''-2\ga}.
\] Then
 \begin{equation*}
%  \label{4.6}
 % \sum_{x\ge X'} \eta_{\vep^{-2}\delta}(x) =
P\Bigl(   \Bigl| \sum_{\scr'\le  x \le \vep^{-1-\ga}}  T_{\vep^{-2}\delta}\eta(x) -
   \vep^{-1}\int_{\vep \scr'}^ {\vep^{-\ga}}G_\delta\rho( r)dr \Bigr|
   \le c \vep^{\alpha''} \vep^{-1-\ga}\Bigr)\;\ge\; 1 - c_k\vep^k,
\end{equation*}
On the other hand by the definition of the quantile
$\scr^{j\delta}(G_\delta\rho)$ and by \eqref{a5.24}
 \begin{eqnarray*}
%  \label{4.6}
\int_{\vep \scr'}^ {\vep^{-\ga}}G_\delta\rho( r)dr&=&\int_{\scr^{j\delta}(G_\delta\rho)}^
{\infty}G_\delta\rho( r)dr-
\int_{\scr^{j\delta}(G_\delta\rho)}^ {\vep \scr'}G_\delta\rho( r)dr-\int_{\vep^{-\ga}}^
{\infty}G_\delta\rho( r)dr
\\& \le&   j\delta  -c' \vep^{\alpha''-2\ga}
\end{eqnarray*}
with $c'=\min\{G_\delta\rho(r):|r-\scr^{j\delta}(G_\delta\rho)|\le 1 \}>0$.
Hence with probability $\ge 1 - c_k\vep^k$,
 \begin{equation*}
%  \label{4.6}
 % \sum_{x\ge X'} \eta_{\vep^{-2}\delta}(x) =
\vep \sum_{\scr'\le  x \le \vep^{-1-\ga}}  T_{\vep^{-2}\delta}\eta(x) \le
j\delta  -c' \vep^{\alpha''-2\ga}
+ c \vep^{\alpha''} \vep^{-\ga} < j\delta-\frac{c'}2 \vep^{\alpha''-2\ga}
\end{equation*}
Let $\scr^b(\eta)$ be the quantile defined in \eqref{lr}
with $b=B_{\vep^{-2}\delta}$, $B$ a Poisson process of rate $j\vep$.
Observing that for any $\kappa>0$
\begin{equation*}
P\Big(\big|\vep B_{\vep^{-2}\delta}-j\delta\big|\le \vep^{\frac 12
  -\kappa}\Big)\ge 1-c_k\vep^k
\end{equation*}
we get that for $\kappa$ small enough and with probability $\ge 1 - c_k\vep^k$,
\begin{equation*}
\vep \sum_{ x \ge \scr^b(T_{\vep^{-2}\delta}\eta)} T_{\vep^{-2}\delta}\eta(x)=\vep B_{\vep^{-2}\delta}
\;\ge\; j\delta- \vep^{\frac 12-\kappa}
\;\ge\;  j\delta-\frac{c'}2 \vep^{\alpha''-2\ga}
\;\ge\;  \vep \sum_{\scr'\le  x \le \vep^{-1-\ga}}  T_{\vep^{-2}\delta}\eta(x)
\end{equation*}
that implies $\scr^b(T_{\vep^{-2}\delta}\eta)\ge R'= \vep^{-1}\scr^{j\delta}(G_\delta\rho) + \vep^{-1+\alpha''-2\ga}$.
Using an analogous argument for the lower bound we get
\begin{equation}
  \label{f56}
  \begin{array}{rcl}
   P(|\vep \scr^b(T_{\vep^{-2}\delta}\eta) -
\scr^{j\delta}(G_\delta\rho)|\le  \vep^{\alpha''-2\ga})&>&1-c_k\vep^k,\\
P(|\vep \scl^a(T_{\vep^{-2}\delta}\eta) -
\scl^{j\delta}(G_\delta\rho)|\le \vep^{\alpha''-2\ga})&>&1-c_k\vep^k;
  \end{array}
\end{equation}
the second inequality is proved by using the same arguments for
$a=A_{\vep^{-2}\delta}$, $A$ being a Poisson process of rate $\vep j$.

\medskip
\noindent{\bf Proof of Proposition \ref{thm7}.} By the definitions
\eqref{g93} and \eqref{r1t},
 \[
\eta^{\vep,\delta,-}_{\vep^{-2}\delta} =
\Gamma^{A_{\vep^{-2}\delta},B_{\vep^{-2}\delta}}T_{\vep^{-2}\delta}\eta,\qquad
\rho^{\delta,-}_\delta= \Gamma^{j\delta}G_\delta\rho.
\]
Since the left and right boundaries after applying $\Gamma$ are the quantiles
before applying it, inequality \eqref{4.6} for the delta$-$ processes follows from
\eqref{4.5} and \eqref{f56}. The same argument applies for the delta$+$
processes.\qed

\vskip.5cm

\subsection
{\bf Hydrodynamic limit of interfaces}
\label{sub5.4}
\nopagebreak

\medskip

\noindent {\bf Proof of Theorem \ref{thm6a}.} We call  $\tau>0$ the time $t$ fixed in
Theorem \ref{thm6a}.
For each $n\in \mathbb N$ we
let $\delta\in \{ \tau 2^{-n} \}$ and consider
the evolutions $(\xi_t^{\vep,\delta,\pm})$ in a bounded time interval, $t\le T=  2^{N+n}\delta=2^N\tau$, $N$ an arbitrary, fixed non negative integer.
%{2.5} we write for $k=0,..,2^{n}-1$
We have by \eqref{2.5},
\begin{eqnarray}
  \label{4.20}
&&\vep \xi^{\vep,\delta,\pm}_{\vep^{-2}(k+1)\delta}(0)-\vep
\xi^{\vep,\delta,\pm}_{\vep^{-2}k\delta}(0)\nonumber\\
&&\qquad= 2 \vep B_{\vep^{-2}(k+1)\delta}
-2\vep B_ {\vep^{-2}k\delta}
+ 2 \vep \sum_{x\ge 0} \eta^{\vep,\delta,\pm}_{\vep^{-2}(k+1)\delta}(x)
-  2 \vep \sum_{x\ge 0} \eta^{\vep,\delta,\pm}_{\vep^{-2}k\delta}(x)
\end{eqnarray}
By \eqref{2.2.1} for $k=1,\dots,2^{n+N}$ and $x>y$,
\begin{equation}
  \label{4.21}
\vep \xi^{\vep,\delta,\pm}_{\vep^{-2}k\delta}(x)-\vep
\xi^{\vep,\delta,\pm}_{\vep^{-2}k\delta}(y)
=  2 {\vep(x-y)}  - 2\vep \sum_{z=y}^{x-1}\eta^{\vep,\delta,\pm}_{\vep^{-2}k\delta}(z)
\end{equation}
By \eqref{4.6a} and \eqref{4.20}--\eqref{4.21}--\eqref{4.22} we then get that
for any $\ga>0$ and any $t\in \{k\delta: k\le 2^{N+n}\}$
\begin{equation}
  \label{4.23}
\lim_{\vep\to 0}P\Big[ \sup_{x\in \mathbb Z} |\vep
\xi^{\vep,\delta,\pm}_{\vep^{-2}t}(x)- \phi^{\delta,\pm}_{t}(\vep x)| \ge
\ga\Big] =0
\end{equation}
In the next section we shall prove that for any $t$:
\begin{equation}
  \label{4.24}
\lim_{n\to \infty} \sup_{r\in \mathbb R} | \phi^{\tau 2^{-n},+}_{t}(r)-\phi^{\tau 2^{-n},-}_{t}(r)| =0
\end{equation}
and that there is a function $\phi^{(\tau)}_t(r)$, $r\in \mathbb R, t\ge 0$, so that
\begin{equation}
  \label{4.25}
\lim_{n\to \infty} \sup_{r\in \mathbb R} | \phi^{(\tau)}_t(r)-\phi^{\tau 2^{-n},-}_{t}(r)| =0
\end{equation}
Then by  \eqref{4.23}, \eqref{4.24}, \eqref{4.25} and \eqref{r111}, for all
$\gamma>0$, 
 \begin{equation}
\label{4.26}
\lim_{\vep\to 0} P\Big[\sup_{x\in \Z} |\vep \xi^{(\vep)}_{\vep^{-2}t}(x)- \phi^{(\tau)}_t(\vep x)|
 \ge \ga\Big]= 0, \quad t\in \mathcal T(\tau)
\end{equation}
where
 \begin{equation}
\label{4.27}
\mathcal T(\tau)=\{k2^{-n}\tau, k\in \mathbb N, n\in \mathbb N\}
\end{equation}
Since $\tau\in \mathcal T(\tau)$ we have proved \eqref{4.00.2} for $t=\tau$ and since $\tau$ was arbitrary,
Theorem~\ref{thm6a} is proved.  \qed

 \setcounter{equation}{0}
\vskip2cm

\section{The macroscopic evolution}
\label{sec:6}

In Subsection \ref{sub6.1} we prove that as $\delta\to 0$ the macroscopic delta
processes converge --that is, we prove \eqref{4.24} and \eqref{4.25}-- and that
$\phi_t$ is well defined by
        \begin{equation*}
\phi_{t}  =\lim_{\delta\to 0} \phi^{\delta,\pm}_{t},\qquad t\ge 0
    \end{equation*}
%that we are going to use later to find properties of $\phi_{t}$.
We also collect
some properties of the macroscopic
evolutions $\phi_t$ and $\rho_t$, in particular monotonicity properties of
$\phi_t$ and existence of boundary points for both motions.
In subsection \ref{sec:7} we construct macroscopic stationary profiles.

\medskip

\subsection{{\bf Existence and regularity of the macroscopic profiles}}  \label{sub6.1}
{\em Proof of \eqref{4.24} and \eqref{4.25}}.
Let $\tau>0$ and $\delta \in \{\tau 2^{-n}, n\in \mathbb N\}$.  We shall first prove by induction on $k$ that for any such $\delta$,
\begin{equation}
  \label{5.1}
 \sup_{r\in \mathbb R} | \phi^{\delta,+}_{k\delta}(r)-\phi^{\delta,-}_{k\delta}(r)| \le j\delta.
\end{equation}
\eqref{5.1} holds for $k=0$ because
\[
\phi^{\delta,+}_{0}(r) = \max\{ \phi_0(r),j\delta +|r|\},\quad \phi^{\delta,-}_{0}(r) =  \phi_0(r).
\]
Suppose next that \eqref{5.1} holds for $k-1$, then by the maximum principle (for the linear heat equation)
and calling $t=(k\delta)^-$,
\[| \phi^{\delta,+}_{t}(r)-\phi^{\delta,-}_{t}(r)| \le j\delta \]
hence \eqref{5.1} holds for $k$ because
\[
\phi^{\delta,+}_{k \delta}(r) = \max\{  \phi^{\delta,+}_{k \delta}(r),j(k+1)\delta +|r|\},\quad
\phi^{\delta,-}_{k \delta}(r) =  \max\{  \phi^{\delta,-}_{k \delta}(r),jk\delta+ |r|\}
\]
 \eqref{5.1} and  \eqref{4.24} are thus proved.

 It is not difficult to see that
     \begin{eqnarray*}
 \phi^{\delta,-}_{t}(r)\le \phi^{\delta',-}_{t}(r),\quad \phi^{\delta ,+}_{t}(r)\ge \phi^{\delta',-}_{t}(r),\qquad \delta=k\delta' {\text { for some integer }}k>0
            \end{eqnarray*}
Thus for any $n\in \mathbb N$ and $t\ge 0$
\begin{equation}
  \label{5.2}
 \phi^{\tau 2^{-n},-}_{t}(r)\le \phi^{\tau 2^{-(n+1)},-}_{t}(r) \le \phi^{\tau 2^{-n} ,+}_{t}(r)
\end{equation}
Hence for any fixed $t$, $\phi^{\tau 2^{-n},-}_{t}(r)$ converges pointwise
to a function that we call $\phi_t^{(\tau)}(r)$ (it may
depend on $\tau$).
Since by definition $|\phi^{\delta,\pm}_t(r)-\phi^{\delta,\pm}_t(r')|\le |r-r'|$,
the convergence is in sup norm and
\eqref{4.25} is then proved with $\phi_t^{(\tau)}(r)$ a Lipschitz function with
Lipschitz constant 1. \qed

\medskip
In the next Theorem we prove   that  $\phi_t^{(\tau)}$ is independent of $\tau$ and also regularity properties of this function.
\vskip.2cm

\begin{theorem}
\label{thm12}
\dpred{ The function  $\phi_t^{(\tau)}$  is independent of $\tau$ and will be denoted by
$\phi_t$ (the same as in Theorem \ref{thm6a}). $\phi_t$ } is continuous in $r$ and $t$,
more precisely there is $c>0$ so that for all $t$, $t'$ such that $|t-t'|\le 1$
and all $r$ and $r'$,
\begin{equation}
  \label{5.5}
|\phi_{t}(r) -\phi_{t'}(r)  |\le c \sqrt{|t-t'|},\;\;\;
|\phi_{t}(r) -\phi_{t}(r ') |\le |r-r'|
\end{equation}
Denoting by $\delta_n$ any sequence of positive numbers such that $\delta_{n+1}=\delta_n/2$ then
\begin{equation}
  \label{5.6}
\phi_{t}  =\lim_{n\to \infty} \phi^{\delta_n,\pm}_{t},\qquad \forall t\ge 0,
\end{equation}
with $\phi^{\delta_n,-}_{t}$ monotonically increasing and $\phi^{\delta_n,+}_{t}$ monotonically decreasing.

\end{theorem}

\begin{proof} Let $t\ge
0$ and $s> 0$, recalling \eqref{k6}  we have
\begin{equation}
\label{5.3}
G_s\phi^{\delta,-}_{t}(r)  \le  \phi^{\delta,-}_{t+s}(r)\le  G_s\phi^{\delta,-}_{t}(r) + j  \dpred{(s+\delta)}.
\end{equation}
The first inequality is obvious.  We have
\begin{eqnarray*}
\phi_{(k+1)\delta}^{\delta,-} &=& \max\{ G_\delta \phi_{k\delta}^{\delta,-}, j(k+1)\delta +|r|\}
\le  \max\{ G_\delta \phi_{k\delta}^{\delta,-}+j\delta, j(k+1)\delta +|r|\} \\&=&
j\delta+ G_\delta  \phi_{k\delta}^{\delta,-},
\end{eqnarray*}
because $\phi_{k\delta}^{\delta,-} \ge jk\delta +|r|$. \dpred{
We then get the last inequality (without the term $j\delta$) for  $t=h\delta$, $h\in \mathbb N$.
If instead $t\in (h\delta, (h+1)\delta)$, then
\[
\phi^{\delta,-}_{(h+1)\delta} \le G_{(h+1)\delta-t} \phi^{\delta,-}_t + j\delta
\]
hence \eqref{5.3}.}

Since $ \phi^{\delta,-}_t(r)$ is Lipschitz, it follows from \eqref{5.3} that
$| \phi^{\delta,-}_{t+s}(r)- \phi^{\delta,-}_{t}(r)|
\le c \sqrt s + j\dpred{ (s+j\delta)}$ and,
by taking $\delta\to 0$,
\begin{equation}
  \label{5.4}
 |\phi^{(\tau)}_{t}(r) -\phi^{(\tau)}_{t'}(r)  |\le c \sqrt{|t-t'|},\;\;\;\text{for all $r$ and $|t-t'|\le 1$}
\end{equation}
We shall next prove that $\phi^{(\tau)}_{t}(r)$ is independent of $\tau$.  Obviously
$\phi^{(\tau)}_{t}(r)=\phi^{(\tau')}_{t}(r)$ if, recalling
\eqref{4.27}, $\tau'\in \mathcal T(\tau)$ (or viceversa). We next suppose that $\tau$ and $\tau'$ are not related in such a way.  We fix $T>0$ and want to prove that $\phi^{(\tau)}_T(r)=\phi^{(\tau')}_T(r)$.
Let $\delta \le \delta'$
%$\delta \ll \delta'$
and $k$ such that $k\delta <\delta' <(k+1)\delta$.
Then
\begin{eqnarray*}
\phi_{\delta'}^{\delta',-}\dpred{(r)} &=& \max\{ G_{\delta'-k\delta}  \phi_{k\delta}^{\delta',-}\dpred{(r)}, j\delta' +|r|\}
\\&\le&
 \max\{ G_{\delta'-k\delta}  \phi_{k\delta}^{\delta,-}\dpred{(r)}+j(\delta'-k\delta), j\delta' +|r|\}
\le \phi_{\delta'}^{\delta,-}\dpred{(r)}+j(\delta'-k\delta)
%j\delta
\end{eqnarray*}
because $\phi_{k\delta}^{\delta,-}\dpred{(r)} \ge jk\delta +|r|$.

By iteration $\phi_{T}^{\delta',-} \le \phi_{T}^{\delta,-} + j N \delta$ if $N$
is the cardinality of
$\{k: k\delta' \le T\}$.  Thus
\[
\phi_{T}^{\delta',-} \le \phi_{T}^{\delta,-} + c T
\frac{\delta}{\delta'}
\]
Take $\delta'= \tau' 2^{-n'}$ and $\delta= \tau 2^{-n}$. Take first $n\to\infty$ and then $n'\to \infty$ to get
$\phi_T^{(\tau')} \le \phi_T^{(\tau)}$.  The opposite inequality holds as well
by interchanging $\delta$ and $\delta'$ in the previous argument.
\qed
\end{proof}

\medskip
\noindent {\bf Proof of Theorem \ref{thm1}.}
From \eqref{4.00.2} we get for all $\gamma>0$
	\begin{equation*}
\lim_{\vep\to0}  P\Bigl( \sup_{a< b}\Bigl|\vep\sum_{\vep x\in[a,b]}\eta^{(\vep)}_{t\vep^{-2}}(x) -\frac 12\big\{\vep(b-a)-  [\phi_t(\vep b)-\phi_t(\vep a)]\big\}\Bigr|>\gamma\Bigr)\;=\;0
\end{equation*}
Since  $\phi_t$ is Lipschitz there is $\rho_t \in L^1$ such that, given any $r_0\in\R$,
\begin{equation}
  \label{5.7}
\phi_{t}(r) =\phi_{t}(r_0) + \int_{r_0}^r \big( 1-2  \rho_t(r') \big)dr'
\end{equation}
and since by \eqref{5.5} the Lipschitz constant is $1$, $\rho_t$ has (almost surely) values in
$[0,1]$ and this proves \eqref{e96}. We prove in Theorem \ref{thm14} later that
$\rho_t\in\cR$, i.e. $-\infty<\scl(\rho_t)\le \scr(\rho_t)<\infty$.
\qed
\medskip

{\noindent{\bf Proof of Theorem \ref{teo2}.}
%Recall that $\rho^\delta_t$ is the same process as $\rho^{\delta,-}_t$.
For any $t$ and $\delta$,
		\begin{equation}
  \label{6.8}
\int_{a}^b \rho^{\delta,-}_t(r) dr-\int_{a}^b \rho_t(r)dr= \dpred{ \frac 12 \Big(
} \phi_t(b)-\phi^{\delta,-}_t(b)+\phi^{\delta,-}_t(a)-\phi_t(a)\Big)
\end{equation}
\dpred{ Let $\delta_n:= 2^{-n}\delta$, then from \eqref{5.6}, for all $r$ }
	\begin{equation}
  \label{6.9}
\phi^{\delta_n,-}_t(r)\le \phi_t(r)\le \phi^{\delta_n,+}_t(r)
\end{equation}
so that from \eqref{5.1} and \eqref{6.8}
	\begin{equation}
  \label{6.10}
\int_{a}^b \rho^{\delta_n,-}_t(r) dr-\int_{a}^b \rho_t(r)dr\le  \dpred{ \frac 12 \Big(} \phi^{\delta_n,+}_t(b)-\phi^{\dpred{\delta},-}_t(b)+\phi^{\delta,-}_t(a)-\phi^{\delta_n,\dpred{-}}_t(a)\Big)\le j\delta
\end{equation}
\qed

\subsection{{\bf Stationary solutions}}
\label{sec:7}
We say that a macroscopic interface $\phi \in V_{0}$ is \emph{stationary} if,
$\phi_0=\phi$ implies $\phi_t = \phi + 2jt$. A macroscopic density
$\rho\in\cR$ is \emph{stationary} if $\rho_0=\rho$ implies
$\rho_t=\rho$. Here $\phi_t$ and $\rho_t$ are the dynamics given by Theorem
\ref{thm6a} and Theorem \ref{thm1}, respectively.

If $\phi$ is stationary, then the density $\rho$
 associated to $\phi$ via \eqref{5.7} is stationary because by \eqref{5.7}
\[
\int_{r_0}^r \big( 1-2  \rho_t(r') \big)dr'= \int_{r_0}^r \big( 1-2  \rho_0(r') \big)dr',\quad
\text{ for all $r_0$, $r$   and $t\ge 0$}
\]

Let
 \begin{equation}
\label{6.2}
 \bar \rho(r) := \begin{cases}
0, &\text{for $r\ge \frac1{4j}$}\\
\frac 12 - 2 jr, &\text{for $|r|\le \frac 1{4j}$}\\
1, &\text{for $r\le -\frac 1{4j}$}
 \end{cases}
% \end{equation}
%  \begin{equation}
% \label{6.3}
\qquad
 \bar \phi(r) :=
 \begin{cases}
2j r^2 + \frac 1{8j} ,&\text{for $|r|\le \frac 1{4j}$}\\
|r|,&\text{for $r\ge \frac 1{4j}$}.
 \end{cases}
\end{equation}

%\vskip.5cm

\begin{theorem}
\label{thm13}

The macroscopic interface $\bar \phi$ and
the associated macroscopic density $\bar \rho$  are stationary.

\end{theorem}

%\medskip

If $\xi^{(\vep)}_0$ approximates $\bar \phi$ in
the sense of \eqref{4.00.1}, then the theorem says that the rescaled interface process as seen from
its vertex converges in the sense of Theorem \ref{thm6a} at any macroscopic time
$t$ to the initial value $\bar \phi$ shifted by $2jt$. An analogous statement
holds for the particle process (but the stationary density profile does not move).

Theorem \ref{thm13} is proven in the next section by introducing a deterministic
\emph{harness process} on $\R^\Z$, a discrete time process that approaches
$\phi_t$ and whose stationary solution is directly computable.

\subsection{\bf Monotonicity}
\label{m88}
We collect some monotonicity properties of the macroscopic interface inherited
from the microscopic dynamics.  We tacitly suppose hereafter that the initial
data $\phi\in V_{o}$, namely that $\phi(r)=V_{o}(r)\equiv |r|$ for all
$|r|$ large enough.

The following lemma is a direct consequence of the definition of $\phi_t$.

\begin{lemma}
\label{lemma16}
For any $t>0$, $\phi_t(r) \ge (|r|+jt)$, $r\in \mathbb R$, and
\[\lim_{|r|\to \infty} |\phi_t(r) - (|r|+jt)|=0
\]
\end{lemma}

\begin{proof}
Let $t = \tau k2^{-n}$.
%and $\delta_m = 2^{-m}, m \ge n$.
It follows from the inequality
$\phi_t^{\tau 2^{-n},-} \le \phi_t \le \phi_t^{\tau 2^{-n},+}$ that
there is $R=R_{n,\tau}$ so that
\[
|r|+ jt \le \phi_t(r)\le |r|+ j(t+ 2^{-n}),\;\; |r|\ge R.
\]
\qed\end{proof}

We shall next establish inequalities
 relating   evolutions with different values of $j$, we thus add
 a superscript $j$ writing
$\phi^{(j)}_t$ and $\phi^{(j,\delta,\pm)}_t$.
%
% Call $T(\phi)$ the vertical shift of $\phi$ such that
%$T(\phi)(r) = |r|$ is asymptotically equal to $|r|$

\medskip

\begin{lemma}
\label{lemma17}

Let $j \le j'$ then for all $t\ge 0$
 \begin{equation}
\label{8.2}
 \phi^{(j)}_t -jt \ge \phi^{(j')}_t -j't.
\end{equation}

\end{lemma}

\medskip

\begin{proof}
 It is enough to prove that
 \[
 \phi^{(j,\delta,-)}_\delta(r) - j\delta \ge
 \psi^{(j',\delta,-)}_\delta (r) -j'\delta,\quad \text{if $\phi\ge \psi$}
 \]
% then $\phi^{(j,\delta,-)}_\delta \ge
%\psi^{(j',\delta,-)}_\delta $.
Let $\scl',\scr'$ be
%where $\psi^{(j,\delta,-)}_\delta(r)= |r|+j'\delta$ and
such that
\[
\psi^{(j',\delta,-)}_\delta(r) = G_\delta \psi(r), \;\;\; \scl'\le r \le
\scr',\quad \text{and $\psi^{(j',\delta,-)}_\delta(r)= |r|+j'\delta$ elsewhere}
\]
By the maximum principle
$
G_\delta \phi \ge  G_\delta \psi
$, then
\[
\phi^{(j,\delta,-)}_\delta(r) \ge G_\delta \phi(r)\ge
\psi^{(j',\delta,-)}_\delta(r), \;\;\; \scl'\le r \le \scr'
\]
and a fortiori:
\[
\phi^{(j,\delta,-)}_\delta(r) \ge \psi^{(j',\delta,-)}_\delta(r) -(j'-j)\delta,
\;\;\; \scl'\le r \le \scr'
\]
By definition $\phi^{(j,\delta,-)}_\delta(r) \ge |r| + j\delta, \;\;\; r \in \mathbb R$ so that
\[
\phi^{(j,\delta,-)}_\delta(r) \ge |r| + j\delta = \psi^{(j',\delta,-)}_\delta(r) -(j'-j)\delta, \;\;\;   r \notin
(\scl', \scr')
\]
which concludes the proof.
\qed\end{proof}

\subsection{{\bf Existence of boundaries}}
\label{sec:9}

Recall the definition in Subsection \ref{d88} of the boundaries
$\scl(\rho),\scr(\rho)$ of a density $\rho\in\cR$.
They are also the boundaries of the interface $\phi_t$ which corresponds to $\rho_t$.

\begin{theorem}
\label{thm14}

The boundaries $\scl(\rho_t),\;\scr(\rho_t)$ of a density
$\rho_t$ as defined in Theorem \ref{thm1} starting from $\rho \in \mathcal R$
 are finite. In other words, $\rho\in\cR$ implies $\rho_t\in\cR$.

\end{theorem}

\begin{proof}
  We shall prove the theorem in the framework of interfaces.  We thus want to
  prove that the boundary points of the interface are finite, that is,
  $-\infty<\scl(\phi^{(j)}_t)$, $\scr(\phi^{(j)}_t)<\infty$ for initial $\rho\in
  \cR$ and all $t\ge0$.  Let ${\bar \phi}^{(j')}$ be the stationary interface
  for the $j'$-evolution.  If $j'<j$ is small enough, $\phi \le {\bar
    \phi}^{(j')}$ so that, by Lemma~\ref{lemma17},
\[
\phi^{(j)}_t -jt \le {\bar \phi}^{(j')}_t - j't= {\bar \phi}^{(j')}.
\]
This implies that the boundary points of $\phi^{(j)}_t$ are bounded by those of ${\bar
  \phi}^{(j')}$ and the  theorem is proved.  \qed
\end{proof}

\section{The harness process}
\label{sec:8} We consider now the (deterministic) Harness
Process proposed by Hammersley \cite{hammersley} with ``moving cone'' boundary
conditions and prove that with the diffusive scaling this process also
converges to the macroscopic evolution $\phi$.

Let $H:\Z\to \R$ and define the operator $\Theta$ by
\begin{equation}
  \label{b27}
  (\Theta H)(x) = \frac{H(x-1)+H(x+1)}{2}
\end{equation}
but to keep notation light we drop the parentheses and write $\Theta H(x)$.
Let $(H_n(x),\,x\in\Z, \,n\in\Z^+)$, $H_n(x)\in\R$ be the deterministic process
satisfying the discrete heat equation:
\begin{equation}
  \label{a7}
 H_{n+1}(x):= \Theta H_n(x) =\Theta^{n+1}H_0(x)
\end{equation}
Here $n$ is time and $x$ is space.

\paragraph{Duality} Let $X^x_n$ be a symmetric nearest neighbors random walk on
$\Z$ with $X^x_0=x$ and $p_n(x,y):= P(X^x_n=y)$ the probability that the walk goes from
$x$ to $y$ in $n$ steps. Then, a simple recurrence shows that
\begin{equation}
  \label{d4}
  H_n(x) =\sum_{y\in\Z} p_n(x,y)H_0(y) = E(H_0(X^x_n))
\end{equation}

\paragraph{Traveling wave solutions} A family of traveling wave solutions of this process
are
\[
\bar H(x) := a x^2 + b,
\]
where $a$ and $b$ are arbitrary constants. Indeed,
% Rescaling:
% \[
% \vep \bar H_{\vep^{-2}t}(\vep^{-1}r) = \vep^2 j r^2 \vep^{-2} + J^2 \vep^{-2}t = jr^2+jt
% \]
% In fact: if $a=2J$
% \[
% \bar H_{n+1}(x) = \frac{a(x-1)^2+a(x+1)^2 + 2an}{2} = a x^2 +a +an = ax^2 +a(n+1)
% \]
% \[
% \vep\bar H_{\vep^{-2}t+1}(r\vep^{-1}) = \vep\frac{J(r\vep^{-1}-1)^2+J(r\vep^{-1}+1)^2 + 2J\vep^{-2}}{2}
% \]
% \[
% =\vep(J r^2\vep^{-2} + J + 2\vep\vep^{-2}t) = j r^2 +j(t+\vep^2)
% \]
\begin{equation}
  \label{a9}
  \bar H_n:=\Theta^n\bar H = \bar H+2an.
\end{equation}

\paragraph{Moving boundaries}
For $v=(v_1,v_2)\in\Z\times \R$, let $V_v:\Z\to\R$ be the cone defined by
$V_{v}(x)=|x-v_1|+v_2$; call $v$ the vertex of $V_v$. Let
\begin{eqnarray*}
%  \label{h21}
  \cH_v&:=&\Big\{K:\Z\to\R\,:\,K(x)=V_v(x)\;\text{for all but a finite number of
  }x\in \Z\Big\}\\
\cH&:=& \cup_{v\in\Z\times\R} \cH_v
\end{eqnarray*}
For $K\in\cH_v$ define
\begin{eqnarray*}
 % \label{b22}
  \scl(K)&:=& \min\{\ell\in\Z\,:\,K(\ell+1)\neq V_v(\ell+1)\}\\
\scr(K)&:=& \max\{\ell\in\Z\,:\,K(\ell-1)\neq V_v(\ell-1)\}
\end{eqnarray*}
The harness process with moving cone boundary conditions and initial
interface $K^\scj_0\in \cH_{(0,0)}$ is defined for $n\ge 1$ by
\begin{equation}
  \label{a8}
K^\scj_{n}(x):=\max\{\Theta K^\scj_{n-1}(x), V_{(0,2\scj n)}(x)\}.
\end{equation}
So that $K^\scj_n\in \cH_{(0,2\scj n)}$.

A travelling wave solution $\bar K$ of this process is associated to $\bar H$:
if the initial interface is given by
\begin{equation}
  \label{b55}
  \bar K^\scj_{0}(x) :=
\left\{
  \begin{array}{ll}
    \frac1{8\scj}+2\scj x^2,&|x|\le \frac1{4\scj}\\
|x|,&|x|\ge\frac1{4\scj}
  \end{array}
\right.
%\qquad \bar K^\scj_{n}(x) := \max\{\Theta \bar K^\scj_{n-1}(x),V_{0,2\scj n}(x)\}
\end{equation}
then the evolution \eqref{a8} at time $n$ gives a translation of the initial
interface:
\begin{equation}
  \label{b57}
\bar K^\scj_{n}(x) = \bar K^\scj_0(x) + 2\scj n  .
\end{equation}
The sides of the cone $y=|x|$ are tangent to the parabola $y=\frac1{8\scj}+2\scj
x^2$ at the points $(\frac{-1}{4\scj },\frac1{4 \scj})$ and
$(\frac{1}{4\scj},\frac1{4 \scj})$.

\paragraph{Hydrodynamic limit} Let $K$ be a Lipschitz function on $\cH$ and define
\begin{equation}
  \label{b62}
\Phi^{(\vep)}_t(r) := \vep \big(K^{\vep j}_{[\vep^{-2}t]}([\vep^{-1}r])\big)
\end{equation}

\begin{proposition} [Hydrodynamic Limit] Let $\phi_t$ be the evolution of
  Theorem \ref{thm6a} with initial condition $\phi$ and let $K^{\vep j}_n$ the
  evolution \eqref{a8} with initial interface $K^{(\vep)}_0(x)= \phi(\vep
  x)$.  Then,
  \label{b56}
  \begin{equation}
    \label{b59}
\lim_{\vep\to0} \sup_{r\in\R}|\Phi^{(\vep)}_t(r) - \phi_t(r)|=0.
  \end{equation}
where the rescaled process $\Phi^{(\vep)}_t(r)$ is defined in \eqref{b62}.
\end{proposition}
The proof of Theorem \ref{thm13} follows from the above proposition:

\medskip
\noindent{\bf Proof of Theorem \ref{thm13} }
Taking $\bar K^{\vep j}$ as the explicit stationary solution in \eqref{b55} with
$\scj=\vep j$ and $\bar\Phi^{(\vep)}$ the corresponding renormalization as in
\eqref{b62},
\begin{equation}
  \label{b58}
\phi_t(r) = \lim_{\vep\to0} \bar\Phi^{(\vep)}_t(r) = \Phi^{(\vep)}_0(r) + 2jt = \bar\phi(r)+2jt
\end{equation}
where the first identity is consequence of \eqref{b59}, the second one comes
from the stationarity of $\bar K$ given by \eqref{b57} and the last one is a
computation based on the explicit expressions of $\bar K^{(\vep)}$ and
$\bar\phi$.
\qed

To prove Proposition \ref{b56} we introduce the delta processes associated to
$K$.

\paragraph{The delta harness processes}
We define the delta harness processes $K^{\scj,\delta,-}_\ell$,
$K^{\scj,\delta,+}_\ell$ as follows. Take $\delta\ge 1$, fix an initial
condition $K^{\scj,\delta,-}_{0}\in\cH_{(0,0)}$,  $K^{\scj,\delta,+}_{0}=
\max\{K^{\scj,\delta,-}_{0},V_{(0,\delta\scj)}\}$ and define iteratively
\begin{equation}
\begin{array}{rclcl}
 % \label{d-}
K^{\scj,\delta,\pm}_{\ell} &:=& \Theta^{\ell-[n\delta]}K^{\scj,\delta,\pm}_{[n\delta]},
&\;\text{if}& \ell\in [[n\delta],[(n+1)\delta]-1],\;n\ge0\\
 K^{\scj,\delta,-}_{[n\delta]} &:=& \max\{  K^{\scj,\delta,-}_{[n\delta]-1},
 V_{(0,n\delta 2\scj)}\},
&\;\text{if}&\;n\ge 1\\
 K^{\scj,\delta,+}_{[n\delta]} &:=& \max\{  K^{\scj,\delta,+}_{[n\delta]-1},
 V_{(0,(n+1)\delta 2\scj)}\},
&\;\text{if}&\;n\ge 1.
\end{array}
\end{equation}
Both processes evolve with \eqref{b27} in the time intervals
$[[n\delta],[(n+1)\delta]-1]\cap \Z$ and update the interface at times
$[n\delta]$: the process delta$-$ takes the max with the cone with vertex $(0,n\delta2\scj)$
while the process
delta$+$ takes the max with the cone with vertex $(0,(n+1)\delta2\scj)$.
The following dominating Lemma follows immediately.

\begin{lemma}
\label{h15}
 If   $K^{\scj,\delta,-}_0\le K^\scj_0\le
  K^{\scj,\delta,+}_0$, then $K^{\scj,\delta,-}_\ell\le K^\scj_\ell\le
  K^{\scj,\delta,+}_\ell$, for all $\ell\ge1$; for all $\delta\ge 1, \scj\ge
  0$. Furthermore,
  \begin{equation}
    \label{h16000}
  \sup_{x\in\Z,\ell\ge 0}\{  K^{\scj,\delta,+}_\ell(x)-
  K^{\scj,\delta,-}_\ell(x)\}\le \delta 2\scj.
  \end{equation}
\end{lemma}

\paragraph{Hydrodynamic limit of the delta processes} Take a macroscopic
initial condition $\phi$ as in Theorem \ref{thm6a}.  Fix $j>0$ and take $\scj=\vep j$, fix $\delta>0$ and
$\vep$ small such that $\delta\vep^{-2}>1$, take $\delta\vep^{-2}$ in the place
of $\delta$ and define
\[
\Phi^{\vep,\delta,\pm}_t(r)
:= \vep K^{\vep j,\vep^{-2}\delta,\pm}_{[\vep^{-2}t]}([\vep^{-1}r])
\]
with initial $K^{(\vep)}_0 (x)=\phi(\vep x)$, which implies
\[
\Phi^{\vep,\delta,\pm}_0(r,0) = \Phi^\vep_0(r) = \phi(\vep[r\vep^{-1}]).%:=\phi^\vep(r)
\]
From Lemma \ref{h15} we have
\begin{equation}
  \label{h16}
  \Phi^{\vep,\delta,-}_t(r)\le \Phi^\vep_t(r)\le\Phi^{\vep,\delta,+}_t(r)
\end{equation}
\begin{equation}
    \label{h17}
\sup_{r,t,\vep}(\Phi^{\vep,\delta,+}_t(r)- \Phi^{\vep,\delta,-}_t(r)) \le 2\delta j.
\end{equation}
\begin{proposition}[hydrodynamics] Let $\phi$ be Lipschitz. Then,
there exists a
constant $C>0$ such that,
\begin{equation}
  \label{h21}
\sup_r|\Phi^{\vep,\delta,\pm}_t(r)- \phi_t^{\delta,\pm}(r)| \le
{Ct}{\delta^{\beta-1}}\vep^{1-2\beta},
\end{equation}
for any $\beta>0$.
\end{proposition}

\begin{proof}
  Take $t<\delta$.  In this case $\phi_t$ obeys the heat equation. Let $W^r_t$
  be Brownian motion with starting point $r$; then the solution $\phi_t$ is
  given by $\phi_t(r) = G_t\phi(r)=
  E\phi(W^r_t)$. By duality \eqref{d4} and assuming $X_n$ and $W_t$ are defined
  in the same probability space,
\begin{eqnarray}
    %\label{g6}
\nonumber    &&|\Phi^{\vep,\delta,\pm}_t(r)- \phi_t^{\delta,\pm}(r)| \;=\;
%\int dr'
%    (G^\vep_t(r,r')\phi(\vep(\vep^{-1}r'))-G_t(r,r') \phi(r')) \\
%&=&
\big|\E  \big(\phi(\vep X_{[\vep^{-2}t]}^{[\vep^{-1}r]})
-\phi(W^r_t)\big)\big|\\
&&\qquad\le\;\E  \big|\phi(\vep X_{[\vep^{-2}t]}^{[\vep^{-1}r]})
-\phi(W^r_t)\big|\;\le\; \E  \big|\vep X_{[\vep^{-2}t]}^{[\vep^{-1}r]}
-W^r_t\big|\label{g24}\\
&&\qquad\le\; C \delta^ {\beta}\vep^{1-2\beta},\qquad \text{for any }\beta>0,\;\;t<\delta,\label{g25}
%&&\qquad\le\; C \vep (\vep^{-2} t)^{1/4} = C \vep^{1/2} t^{1/4}\label{g25}
\end{eqnarray}
where in \eqref{g24} we used that $\phi$ is Lipschitz and in \eqref{g25} the
dyadic KMT coupling between the Brownian motion and the random walk
 \cite{LL}, Theorem 7.1. %  \paf{This is enough, but with the dyadic
  % coupling in chapter 7 of LL we can get $C\vep (\log(t/\vep^2))$}.

At $t=\delta$ we have
\begin{eqnarray}
  \label{g26}
      |\max\{\Phi^{\vep,\delta,\pm}_{\delta-},V_{(0,[\vep^{-2}\delta]2j\vep^2)}\}-
      \max\{\phi_{\delta-}^{\delta,\pm},V_{(0,\delta2j)}\}| \le C \delta^{\beta}\vep^{1-2\beta}
      +\vep,
\end{eqnarray}
%\paf{$|\max(a,b)-\max(c,d)|\le |\max(a,b)-\max(c,b)|+|\max(c,b)-\max(c,d)| \le |a-c|+|b-d|$}
because the two cones differ at most by $\vep$. Changing the constant $C$,
\eqref{g26} is bounded by $C \delta^{\beta}\vep^{1-2\beta}$, so that  iterating
\eqref{g26} $[(t+1)/\delta]$ times we get \eqref{h21}. \qed
\end{proof}

\medskip
\noindent{\bf Proof of Proposition \ref{b56}}
As a consequence of \eqref{h16} and \eqref{h17},
\begin{eqnarray*}
  %\label{h18}
  |\phi_t-\Phi^{(\vep)}_t(r)| &\le&
  |\phi_t-\phi_t^{\delta,\pm}|+|\phi_t^{\delta,\pm}-\Phi^{\vep,\delta,\pm}_t|
  +|\Phi^{\vep,\delta,\pm}_t - \Phi_t^{(\vep)}|\nonumber\\
&\le& 2\delta  + Ct\delta^{1-\beta} \vep^{1-2\beta} + 2\delta.
\end{eqnarray*}
Taking first $\vep\to0$ and then $\delta\to0$, we get \eqref{b59}.\qed

\section{Conclusions}
\label{sec:80}
In this section we summarize the results we have obtained so far.  In Theorem \ref{thm6a}
and Theorem \ref{thm1} we have proved convergence in
the hydrodynamic limit to a deterministic evolution
for the interface and, respectively, the density.
The limit interface $\phi_t$ is Lipschitz continuous in space and
continuous in time, see Theorem \ref{thm12}.
At each time $t>0$ it ``belongs'' to a cone, in the sense that there
are real numbers $\scl(\phi_t)$ and $\scr(\phi_t)$ so that $\phi_t(r)= |r| + jt$ for $r\notin
(\scl(\phi_t),\scr(\phi_t))$, Theorem \ref{thm14}.  The limit particle densities
inherit analogous properties from the interface.

The interface evolution $\phi_t$ is characterized in terms of a sequence
of upper and lower bounds $\phi_t^{\delta,\pm}$ which in the limit $\delta\to 0$
become identical. $\phi_t^{\delta,\pm}$ are solutions of time-discrete
Stefan problems in the sense that they are obtained
by alternating linear heat diffusion and motion of the boundaries.

We miss however a proof that the limit evolution satisfies
the Stefan problem described in the introduction for the particle density.
We do know however that the stationary solution of \eqref{0.1}-\eqref{0.2}
is indeed stationary for the limit evolution, Theorem \ref{thm13}.   The formula for the velocity
of the boundaries is quite natural once we observe that the levels
of the solution of the heat equation have velocity $-\rho''_t/(2\rho'_t)$ ($\rho'$
and $\rho''$ the space derivatives of $\rho$).  To get \eqref{0.2} we need to add
the information that at the endpoints $\rho' = -2j$ which is consistent with the analysis
of the stationary solution.

We have proved in Theorem \ref{thm1.2} that there is a stationary measure for
the particle process at fixed $\scj=j\vep>0$; we miss however a proof that in
the limit $\vep\to 0$ it becomes supported by the stationary solution of
\eqref{0.1}-\eqref{0.2}, even though this is stationary for the limit evolution.

%
%
%
%\paragraph{Open problems}
%There are basic questions however which remain
%essentially open:
%\begin{itemize}
%\item  Identify the hydrodynamic equations (which  $\rho_t$ and $\phi_t$ satisfy).
%\item Prove (at least for smooth initial data and small times) that $\rho_t$ is
%  equal to 1 and 0 outside of a compact and, more ambitiously, that there is an
%  interval $(\scl(\rho_t),\scr(\rho_t))$ with $\scl(\rho_t)$ continuous as
%  functions of $t$, functions such that $\rho_t$ has values in $(0,1)$ in the
%  interval while it is identically 1 and 0 respectively to the left and to the
%  right of the interval.
%\item Identify the stationary density ($\rho_t=\rho_0$  for all $t\ge 0$) and relate
%them to the hydrodynamic limit of the stationary measure for the particle process.
%\end{itemize}
%

\medskip
{\bf Acknowledgments.}
We thank Stefano Olla for many useful discussions.  \dpred{ We are also indebted to F. Comets and
H. Lacoin for helpful comments and discussions and to a referee of PTRF for useful comments.}

The research has been partially supported by PRIN 2009 (prot. 2009TA2595-002).
A. De Masi thanks the  Departamento de Matem\'atica Universidad de Buenos Aires
for  support and hospitality.
P.A. Ferrari thanks warm hospitality and support at Dipartimento di Matematica Universit\`a
di Roma Tor Vergata and Universit\`a de L'Aquila.

\vskip.5cm

\bibliography{dfp}{}

\begin{thebibliography}{10}

\bibitem{CDGP}
Gioia Carinci, Anna~De Masi, Cristian Giardin\`a, and Errico Presutti.
\newblock Hydrodynamic limit in a particle system with topological
  interactions.
\newblock {\em arXiv:1307.6385}, 2013.

\bibitem{DFL}
A.~De~Masi, P.~A. Ferrari, and J.~L. Lebowitz.
\newblock Reaction-diffusion equations for interacting particle systems.
\newblock {\em J. Statist. Phys.}, 44(3-4):589--644, 1986.

\bibitem{DPTV1}
A.~De~Masi, E.~Presutti, D.~Tsagkarogiannis, and M.~E. Vares.
\newblock Current reservoirs in the simple exclusion process.
\newblock {\em J. Stat. Phys.}, 144(6):1151--1170, 2011.

\bibitem{demasipresutti}
Anna De~Masi and Errico Presutti.
\newblock {\em Mathematical methods for hydrodynamic limits}, volume 1501 of
  {\em Lecture Notes in Mathematics}.
\newblock Springer-Verlag, Berlin, 1991.

\bibitem{DPT}
Anna De~Masi, Errico Presutti, and Dimitrios Tsagkarogiannis.
\newblock Fourier law, phase transitions and the stationary {S}tefan problem.
\newblock {\em Arch. Ration. Mech. Anal.}, 201(2):681--725, 2011.

\bibitem{DPTV2}
Anna De~Masi, Errico Presutti, Dimitrios Tsagkarogiannis, and Maria~E. Vares.
\newblock Truncated correlations in the stirring process with births and
  deaths.
\newblock {\em Electron. J. Probab.}, 17:no. 6, 35, 2012.

\bibitem{Durrett}
Rick Durrett and Daniel Remenik.
\newblock Brunet-{D}errida particle systems, free boundary problems and
  {W}iener-{H}opf equations.
\newblock {\em Ann. Probab.}, 39(6):2043--2078, 2011.

\bibitem{hammersley}
J.~M. Hammersley.
\newblock Harnesses.
\newblock In {\em Proc. {F}ifth {B}erkeley {S}ympos. {M}athematical
  {S}tatistics and {P}robability ({B}erkeley, {C}alif., 1965/66), {V}ol. {III}:
  {P}hysical {S}ciences}, pages 89--117. Univ. California Press, Berkeley,
  Calif., 1967.

\bibitem{L}
Hubert Lacoin.
\newblock The scaling limit of polymer pinning dynamics and a one dimensional
  {S}tefan freezing problem.
\newblock {\em arXiv:1204.1253}, 2012.

\bibitem{LandimValle}
Claudio Landim and Glauco Valle.
\newblock A microscopic model for {S}tefan's melting and freezing problem.
\newblock {\em Ann. Probab.}, 34(2):779--803, 2006.

\bibitem{LL}
Gregory~F. Lawler and Vlada Limic.
\newblock {\em Random walk: a modern introduction}, volume 123 of {\em
  Cambridge Studies in Advanced Mathematics}.
\newblock Cambridge University Press, Cambridge, 2010.

\bibitem{rost}
H.~Rost.
\newblock Nonequilibrium behaviour of a many particle process: density profile
  and local equilibria.
\newblock {\em Z. Wahrsch. Verw. Gebiete}, 58(1):41--53, 1981.

\end{thebibliography}
\bibliographystyle{plain}

\end{document}